\documentclass[11pt]{article}
\usepackage[title]{appendix}
\usepackage{epsfig, amsmath, amssymb, amsthm, times}
\usepackage{color}

\newtheorem {thm}{Theorem}[section]
\newtheorem {lem}[thm]{Lemma}

\theoremstyle{defintion}

\theoremstyle{assumption}
\newtheorem {assume}[thm]{Assumption}
\theoremstyle{remark}
\newtheorem{rem}[thm]{Remark}
\theoremstyle{example}
\newtheorem{ex}[thm]{Example}

\def\p{\partial}

\def\E{\operatorname{\mathbb E}}
\def\P{\operatorname{\mathbb P}}

\def\R{{\mathbb R}}
\def\N{{\mathbb N}}
\def\Z{{\mathbb Z}}

\def\iS{\mathcal{S}}
\def\B{\mathcal{B}}

\def\lbl{\label}
\def\be{\begin{equation}}
\def\ee{\end{equation}}
\def\p{\partial}
\def\Gn{\Gamma^n}

\def\1{\operatorname{\mathbf 1}}
\newcommand{\iu}{{i\mkern1mu}}

\begin{document}

\title{The large deviation principle for interacting dynamical systems on
random graphs}
\author{ Paul Dupuis \thanks{
Division of Applied Mathematics, Brown University, \texttt{%
Paul\_Dupuis@brown.edu},
 research supported in part by the NSF (DMS-1904992)}\; and Georgi S. Medvedev \thanks{%
Department of Mathematics, Drexel University, 3141 Chestnut Street,
Philadelphia, PA 19104, \texttt{medvedev@drexel.edu},
research supported in part by the NSF (DMS-2009233)}
}
\maketitle
\begin{abstract}
  Using the weak convergence
  approach to large deviations,
  we formulate and prove the large deviation principle (LDP) for W-random graphs
  in the cut-norm topology. This generalizes the LDP for Erd\H{o}s-R{\' e}nyi
  random graphs by Chatterjee and Varadhan. 
  Furthermore, we translate the LDP for random graphs to
  a class of interacting dynamical systems on such graphs. To this end, we demonstrate
  that the solutions of the dynamical models depend continuously on the underlying
  graphs with respect to the cut-norm and apply the contraction principle.
  \end{abstract}

\section{Introduction}\label{sec.intro}
\setcounter{equation}{0}

The problem of the macroscopic description of  motion of interacting particles has a long history 
\cite{Gol16, Jab14}. When the number of particles is large, the analysis of individual trajectories
becomes intractable and one is led to study statistical distribution of particles in the phase
space. This is done using the Vlasov equation or other kinetic equations describing the
state of the system in the continuum limit as the size of the system goes to infinity
\cite{Dob79, Neu78, BraHep77}. Modern applications ranging from neuronal networks to
power grids feature models with spatially structured interactions. The derivation of the continuum
limit for such models has to deal with the fact that in contrast to 
the classical setting used in \cite{Dob79, Neu78, BraHep77}, the particles are no longer identical,
and it has also to take into account the limiting connectivity of the network assigned by the underlying graph
sequences. This problem was addressed in \cite{Med14a, Med14b}, where the ideas from the
theory of graph limits \cite{LovGraphLim12} were used to formulate and to justify the
continuum limit for interacting dynamical systems on certain convergent graph sequences.
In particular, in  \cite{Med14a} and in the followup paper \cite{Med19}, solutions of coupled
dynamical systems on a sequence of W-random graphs were approximated by those of a
deterministic nonlocal diffusion equation on a unit interval, representing a continuum of nodes in
the spirit of the theory of graph limits. This result can be interpreted as a Law of Large Numbers
for the solutions of the initial value problems (IVPs) of interacting dynamical systems on W-random graphs.
In the present paper, we study the accuracy of the continuum limit for coupled dynamical systems
on W-random graphs at the level of large deviations, i.e., we are interested in exponentially small probabilities
of $O(1)$ deviations of the solutions of the discrete system from their typical behavior.
This is the main goal of the work.

Motivated by applications, we consider the following model of $n$ interacting particles
\begin{eqnarray}
  \label{frame}
  \dot u_i^n &=& f(u_i^n, \xi_i^n, t) +\frac{1}{n}\sum_{j=1}^n X^n_{ij}D(u_i^n, u^n_j),\\
  \label{frame-ic}
  u_i^n(0) &=& g^n_i,\quad i\in [n]:=\{1,2,\dots, n\},
\end{eqnarray}
where $u^n_i:\R^+\rightarrow \mathbb{X}$ stands for the state of particle~$i$, $f$ describes its intrinsic dynamics,
and $D$ models the pairwise interactions between the particles. The network connectivity is defined by the
graph $\Gamma^n$ with the adjacency matrix $\{X^n_{ij}\}$. The phase space $\mathbb{X}$ can be
either $\R,$ or $\R/\Z,$ or $\R^d$ depending on the model at hand. Parameters $\xi^n_i\in\R^p$ and
initial conditions $g^n_i,$ $i\in [n],$ are random in general.
Many network models in science and technology fit into the framework \eqref{frame}, \eqref{frame-ic}.
Examples include neuronal networks, power grid models, and various coupled oscillator systems to name
a few. We present the Kuramoto model of coupled phase oscillators \cite{Kur84-book}, as an illustrative example.

\begin{ex}\label{ex.Kuramoto} In the Kuramoto model, $\mathbb{X}=\R/\Z$, $f(u,\xi,t)=\xi$, and
  $D(u,v)=\sin\left(2\pi(v-u)\right)$.
  $\{\xi^n_i\}$ are independent and identically distributed (iid) random variables.
  This model was studied using Erd\H{o}s-R{\' e}nyi, small-world, and power law random graphs
  (see \cite{Med19} and references therein).
  \end{ex}

  % Clearly, one can expect a limiting behavior of the discrete model \eqref{frame}, \eqref{frame-ic}
  % for large $n$ only if the underlying graphs 
  % $(\Gamma^n)$ possess well-defined asymptotics.
  In this paper, we use W-random graphs to define network connectivity in \eqref{frame}.
  This is a flexible framework for modeling random graphs \cite{LovSze06}, which fits
  seamlessly into the analysis of the continuum limit of interacting dynamical systems
  like \eqref{frame} \cite{Med19}.
  Specifically,  given $W\in \iS=\{ U\in L^\infty([0,1]^2):\; 0\le U\le 1\}$, which prescribes
  the  asymptotic behavior of $\{\Gamma^n\}$,  we
define $\{X^n_{ij}, \; (i,j)\in [n]^2\}$ as independent random variables 
such that
\begin{equation}\label{Pedge}
\P\left(X^n_{ij}=1\right) =W^n_{ij}\quad\mbox{and}\quad \P\left(X^n_{ij}=0\right) =1-W^n_{ij},
\end{equation}
where
\begin{equation}\label{Wnij}
W^n_{ij}= n^2 \int_{Q^n_{ij}} W(x,y) dxdy,
\quad Q^n_{ij}=Q^n_i\times Q^n_j,
\;\; Q_i^n=\left[ \frac{i-1}{n}, \frac{i}{n}\right)\quad i,j\in [n].
\end{equation}

  For large $n$ the direct analysis of \eqref{frame}, \eqref{frame-ic} is not feasible
  and one is led to seek other ways. A common alternative to studying individual trajectories
  of \eqref{frame}, \eqref{frame-ic} is to consider a continuum limit as the size of the system tends to
  infinity. 
  In this case, under the suitable assumptions on $f,$ $D,$ and $W$ the discrete model \eqref{frame},
  \eqref{frame-ic} can be approximated by the following continuum limit (cf.~\cite{Med19}):
\begin{eqnarray}\lbl{cKM}
\p_t u(t,x)& = &f\left(u(t,x), t\right)+ \int W(x,y) D\left(u(t,x), u(t,y)\right) dy,\\
\lbl{cKM-ic}
u(0,x)&=&g(x).
\end{eqnarray}
Here and below, if the domain of integration is not specified, it is implicitly assumed to be  $[0,1].$
Also, we have dropped the dependence on $\xi$  and let $\mathbb{X}=\R$ to simplify the presentation. At the end of the paper,
we comment on how to extend the analysis to cover models depending on random parameters.
Until then we will study the following discrete model:
\begin{eqnarray}
  \label{KM}
  \dot u_i^n &=& f(u_i^n, t) +n^{-1}\sum_{j=1}^n X^n_{ij}D(u_i^n, u^n_j),\\
  \label{KM-ic}
  u_i^n(0) &=& g^n_i,\quad i\in [n],
\end{eqnarray}
where $u^n_i:\R^+\rightarrow \R$ and the rest  is the same as in \eqref{frame}, \eqref{frame-ic}.

Let $C\left([0,T], L^2([0,1])\right)$ stand for the space of continuous vector-valued functions
$[0,T]\ni t\mapsto u(t,\cdot)\in L^2([0,1])$ equipped with the norm (cf.~\cite{LioMag72})
$$
\|u\|_{C\left([0,T], L^2([0,1])\right)}=\sup_{t\in [0,T]} \|u(t,\cdot)\|_{L^2([0,1])}.
$$
To compare solutions of the discrete and continuous models, we represent the former as an element of
$C\left([0,T], L^2([0,1])\right)$:
\begin{equation}\label{rewrite-un}
  u^n(t,x):=\sum_{i=1}^n u^n_i(t) \1_{Q^n_i}(x),
\end{equation}
where $\1_A$ stands for the indicator function of $A$.
Then for the model \eqref{KM} with a sequence of W-random graphs \eqref{Pedge}, \eqref{Wnij}
with deterministic initial conditions
it was shown in \cite{Med19} that
$$
\lim_{n\to \infty} \|u^n-u\|_{C\left([0,T], L^2([0,1])\right)}=0\quad \mbox{a.s.}
$$
This statement can be interpreted as the Law of Large Numbers (LLN) for \eqref{KM}.
Note that the continuum limit \eqref{cKM} is deterministic, while the discrete models \eqref{KM}
are posed on random graphs.
Therefore, the solution of \eqref{cKM}, \eqref{cKM-ic} presents the typical behavior of the 
solutions of the discrete system \eqref{KM}, \eqref{KM-ic} on random graphs for large $n$.
We are interested in the deviations of $u^n$ from this typical behavior.
Below we formulate and prove an LDP for solutions of the discrete model \eqref{KM}, \eqref{KM-ic}.
Before we address this problem, we first establish an LDP for a sequence of W-random graphs
$\{\Gamma^n\}$. To this end,  we represent them as elements of $\iS$ through
\begin{equation}\label{def-Fn}
H^n=\sum_{i,j=1}^n X^n_{ij}\1_{Q^n_{ij}},
\end{equation}
where $\{X_{ij}^n\}$ is the adjacency matrix of $\Gn$.
Then using the weak convergence method \cite{BudDup19}, one can show
(see Theorem \ref{thm.Wrandom}) that $\{H^n\}$, or to be more precise a 
sequence of equivalence classes for which each $H^n$ provides a representative element, satisfies an LDP.
% with scaling
% sequence $n^2$ and the
% rate function
% \begin{equation}\label{Ups}
%   \begin{split}
%   \Upsilon(V) &=\int_{[0,1]^2}  R \left( \{ V(y), 1-V(y) \} \Vert \{ W(y), 1-W(y) \}  \right) dy\\
% &=\int_{[0,1]^2} \left\{  V(y) \log\left( {V(y)\over W(y)}\right) +
%   \left( 1-V(y) \right) \log\left( {1-V(y)\over 1-W(y)} \right) \right\} dy,
% \end{split}
% \end{equation}
% where $R(\theta\Vert\mu)$ is the relative entropy of probability measures $\theta$ and $\mu$, i.e., 
% \begin{equation*}
% R\left( \theta\left\Vert \mu\right.\right)=\int \left(\log {\frac{d\theta}{%
% d\mu }}\right)d\theta
% \end{equation*}
% if $\theta \ll \mu$ and $R\left( \theta\left\Vert \mu\right.\right)=\infty$ otherwise.
This LDP for W-random graphs 
% with the rate function \eqref{Ups} 
generalizes the LDP
for Erd\H{o}s-R{\' e}nyi graphs in \cite{ChaVar11}  
and gives logarithmic asymptotics.
The LDP is established using the same cut norm topology on $\iS$ as in \cite{ChaVar11},
which turns out to be suitable for our later application to dynamical models.
% Using a topology on $S$ that is suitable for our later application to dynamical models, if $O\subset S$ is open then 
% \[
% \liminf_{n \rightarrow \infty}\frac{1}{n^2}\log P(\tilde W^n)
% \geq - \inf_{V \in O}\Upsilon(V),
% \]
% and if $F\subset S$ is closed then 
% \[
% \limsup_{n \rightarrow \infty}\frac{1}{n^2}\log P(\tilde W^n)
% \leq - \inf_{V \in F}\Upsilon(V).
% \]
\begin{rem}\label{rem.undirected}
By construction, $\{\Gamma^n\}$ is a sequence of random directed graphs. The definition of
$\Gamma^n$ can be easily modified  if graphs $\Gamma^n$ are assumed to be undirected instead.
To this end, the entries $X^n_{ij}, \; 1\le i\le j\le n,$ are defined as above and the rest
are defined by symmetry: $X^n_{ij}=X^n_{ji},$ $1\le j<i < n.$ 
Also, the scaling sequence and the rate function used in Theorem \ref{thm.Wrandom} need to be modified accordingly.
\end{rem}

To translate the LDP for W-random graphs to the space of solutions of \eqref{cKM}, we
use the contraction principle \cite{BudDup19}. To this end, we need to show that the solutions of the
IVP for \eqref{cKM} depend continuously on $W\in \iS$ in the appropriate topology.
It would be natural and easier to establish an LDP for $\{\Gn\}$ in
the weak topology. However, the weak topology is not enough to construct the 
continuous mapping from $\iS$ to $C\left([0,T], L^2([0,1])\right)$, the space of solutions of
\eqref{cKM}, \eqref{cKM-ic}. Conversely, the strong topology, which would
guarantee the continuous dependence, is too discriminate. Random graph sequences, like
Erd\H{o}s-R{\' e}nyi graphs, do not converge in the $L^2$-norm. This suggests that like in
combinatorial problems involving random graphs (cf.~\cite{Cha17}), the right topology for the
model at hand is that generated by the cut--norm. On the one hand it metrizes
graph convergence \cite{LovGraphLim12}, i.e., the random graph sequence used in \eqref{KM}
converges in the cut-norm. On the other hand, the cut-norm is strong enough
to provide continuous dependence of solutions of \eqref{cKM}, \eqref{cKM-ic} on $W$.
It is these considerations that motivate the use of the cut topology in the 
problem of large deviations for random
graphs.

This work fits into two partially independent lines of research. On the one hand, there has been an
interest in developing the theory of large deviations for large random graphs. This
research is motivated by questions in combinatorics. Recently, building on the
results of the theory of graph limits Chatterjee and Varadhan proved an LDP for Erd\H{o}s-R{\' e}nyi
random graphs in cut norm topology \cite{ChaVar11}. Our Theorem~\ref{thm.Wrandom} generalizes the LDP of
Chatterjee and Varadhan
to W-random graphs, a large class of random graphs. We use the weak convergence techniques for large deviations
\cite{BudDup19}, which afford a short proof of the LDP. On the other hand, there has been a search for rigorous methods
for studying large networks of interacting dynamical systems. This research is motivated by problems
in statistical physics, which has been reinvigorated by widespread presence of networks in modern science. 
The continuum limit is one of the main tools for analyzing dynamics of large networks. The main result of this
paper provides fine estimates of the accuracy of the continuum limit approximation developed in \cite{Med19}
for a large class of models on W-random graphs.
% Our large deviation estimates show explicit
% relation between the connectivity of the network and the variability of its dynamics.
Previous studies of large deviations for interacting dynamical systems on random graphs like \eqref{KM}
considered models forced by white noise. For such models in \cite{OliRei19, CopDieGia20} it was shown
that if a spatially averaged model satisfies an LDP with respect to white noise forcing then so will the original
model on the random graph. This does not address large deviations due to random connectivity.
The rate function derived in this paper gives an explicit relation between
the random connectivity of the network and the variability of the network dynamics, which is often sought in
applications.
% The large deviations due to stochastic forcing can be analyzed in the same manner as
%was done for initial conditions (cf.~Theorem~\ref{thm.model-I}). We did not pursue this in the present paper to keep
%the presentation simple.\footnote{PD - I would leave this out unless 100\% sure}
% \textit{Here also need to relate
%  to the existing results on LDs for networks of interacting dynamical systems
%  e.g. \cite{OliRei19, PraHol96, BDF12, CopDieGia20}.}

After this paper was submitted for publication, there has been progress on large deviations for block and
step graphon random graph models \cite{Borgs2020LDs, Grebik2021LDs}. Using our terminology,
these papers present LDPs for W-random graph sequences, for which the graphon $W$ is a step function.
The rate functions and the scaling sequences derived in these papers are consistent with our results for
bounded graphons. The proofs of the LDPs in \cite{Borgs2020LDs, Grebik2021LDs} are built upon the
method of Chatterjee and Varadhan \cite{ChaVar11, Cha17}. We rely on the weak convergence techniques
\cite{BudDup19}.

The outline of the paper is as follows.  In the next section we formulate the assumptions on the
model and impose random initial conditions. In Section~\ref{sec.graphon}, we review certain facts
from the theory of graph limits \cite{LovGraphLim12}, which will be used in the main part of the paper.
In Section~\ref{sec.ldp},
we formulate the LDPs for the combinatorial and dynamical problems. In Section~\ref{sec.proof},
we prove the LDP for W-random graphs.
% The proof uses the weak convergence approach to large
% deviations \cite{BudDup19}.
In Section~\ref{sec.contract}, we establish the contraction principle relating
the LDPs for the combinatorial and dynamical models. Certain extensions of the main result
are discussed in Section~\ref{sec.generalize},
and the lower semicontinuity of the rate function is proved in a concluding appendix.

\section{The model}\label{sec.model}
\setcounter{equation}{0}
In this section, we formulate our assumptions on the dynamical model \eqref{KM}, \eqref{KM-ic},
except for assumptions on $\{X^n_{ij}\}$,
which were given in \eqref{Pedge}, \eqref{Wnij}.
% For convenience, we rewrite the discrete model with the simplified
% $f$ that has no dependence on random parameters:
% \begin{eqnarray}
% \dot{u}_{i}^{n}
% &=&f(u_{i}^{n},t)+n^{-1}\sum_{j=1}^{n}a_{ij}^{n}D(u_{i}^{n},u_{j}^{n}),
% \label{KM} \\
% u_{i}^{n}(0) &=&g_{i}^{n},\quad i\in \lbrack n]. \label{KM-ic}
% \end{eqnarray}%
Functions $f$ and $D$ describe the intrinsic dynamics of individual
particles and interactions between two particles at the adjacent nodes of $%
\Gamma ^{n}$ respectively. We assume that $f:\mathbb{R}^2\rightarrow \mathbb{R}$ is  bounded, and uniformly Lipschitz
continuous in $u$ in that
\begin{equation}
|f(u,t)-f(v,t)|\leq L_{f}|u-v|\quad u,v\in {\mathbb{R}},\;t\in {\mathbb{R}},
\label{Lip-f}
\end{equation}%
and  continuous in $t$ for each fixed $u$.
$D$ is a bounded and Lipschitz continuous function: 
\begin{equation}
|D(u,v)-D(u^{\prime },v^{\prime })|\leq L_{D}\left( |u-u^{\prime
}|+|v-v^{\prime }|\right) .  \label{Lip-D}
\end{equation}%
By rescaling time in \eqref{KM} if necessary, one can always achieve
that $f$ and $D$ are bounded by $1$. Thus, we assume 
\begin{equation}
|f(u,t)|\leq 1\quad \mbox{and}\quad |D(u,v)|\leq 1.  \label{D-bound}
\end{equation}
Finally, for the contraction principle in Section~\ref{sec.contract} we will
need in addition to assume that $D\in H_{\mathrm{loc}}^{s}({\mathbb{R}}^{2}),s>1,$
where $H^{s}_{\mathrm{loc}}$ stands for the Sobolev space of
functions on $\R^2$ that are
square integrable together with their generalized derivatives up to
order $s$ on any compact subset of $\R^2$.

We now turn to the  initial condition. Let $\B=L^{2}([0,1])$ with the usual norm and
associated topology. Assume that $\{G^{n}\}$ is a sequence of $\B$-valued
random variables that are independent of $\{X_{ij}^{n}\}$
and that satisfy an LDP with function $K$ and scaling sequence $n^2$. To define an initial condition
for the discrete system, we let 
\begin{equation*}
g_{i}^{n}=n\int_{Q_{i}^{n}}G^{n}(y)dy\text{ for }x\in Q_{i}^{n}.
\end{equation*}%

Suppose that $\bar{G}^{n}$ is defined by $\bar{G}^{n}(x)=g_{i}^{n}$ for $x\in Q_{i}^{n}$.
%\marginpar{GM: should $\bar{G}^{n}(x)$ be $g_{i}^{n}$ instead of $g_{i}^{n}/n$?}
Then
it is not automatic that $\{G^{n}\}$ and $\{\bar{G}^{n}\}$ have the same
large deviation asymptotics, and so we impose the following.

\begin{assume}
\label{assum:LDIC}$\{\bar{G}^{n}\}$ satisfies the LDP
in $\B$ with the rate function $K$ and scaling sequence $n^2$.
\end{assume}

\begin{rem}
As an alternative condition we could have simply assumed a large deviation property of $\{\bar{G}^{n}\}$.
However, it seems easier to pose conditions on $\{{G}^{n}\}$ under which an LDP holds than to pose conditions
on $\{{g}^{n}_i,i\in [n]\}$.
\end{rem}

\begin{ex}\label{ex.ic-1}
If $G^n=g$ for some fixed deterministic $g \in B$ then Assumption \ref{assum:LDIC}
holds with $K$ defined by $K(h)=0$ if $h=g$ and $K(h)=\infty$ otherwise.
Indeed, in this case we have 
\begin{align*}
0   \leq\left\Vert g-\bar{G}^{n}\right\Vert^2 _{L^{2}([0,1])}
 &=\sum_{i=1}^{n}\int_{Q_{i}^{n}}\left[  g(x)-n\int_{Q_{i}^{n}}g(y)dy\right]
^{2}dx\\
& =\int_{[0,1]}g(x)^{2}dx-\frac{1}{n}\sum_{i=1}^{n}\left[  n\int_{Q_{i}^{n}%
}g(y)dy\right]  ^{2}.
\end{align*}
Letting $f_{n}(x)=n\int_{Q_{i}^{n}}g(y)dy$ for $x\in Q_{i}^{n}$ we find
$f_{n}(x)\rightarrow g(x)$ a.e. with respect to Lebesgue measure, and thus by
Fatou's lemma%
\[
\liminf_{n\rightarrow\infty}\frac{1}{n}\sum_{i=1}^{n}\left[  n\int_{Q_{i}^{n}%
}g(y)dy\right]  ^{2}\geq\int_{\lbrack0,1]}g(x)^{2}dx.
\]
Hence $\left\Vert g-\bar{G}^{n}\right\Vert _{L^{2}([0,1])}\rightarrow0$.
\end{ex}

\begin{ex}\label{ex.ic-2}
Suppose that there is $M<\infty $ such that $G^{n}$ is Lipschitz
continuous with constant $M$ almost surely (a.s.). Then 
\begin{align*}
\left\Vert G^{n}-\bar{G}^{n}\right\Vert _{L^{2}([0,1])}^{2}&
=\sum_{i=1}^{n}\int_{Q_{i}^{n}}\left[ G^{n}(x)-n\int_{Q_{i}^{n}}G^{n}(y)dy%
\right] ^{2}dx \\
& \leq \sum_{i=1}^{n}\int_{Q_{i}^{n}}\left[ M\frac{1}{n}\right] ^{2}dx \\
& =\ M^{2}\frac{1}{n^{2}}  \rightarrow 0.
\end{align*}%
Since the convergence is uniform in $\omega$
$\{\bar{G}^{n}\}$ satisfies the same LDP as $\{{G}^{n}\}$, and therefore Assumption \ref{assum:LDIC}
holds.
Note that the Lipschitz condition is stronger than needed. For instance, it can be relaxed to
requiring that $G^n$ belongs to a generalized Lipschitz space \cite[Lemma~5.2]{KVMed19}.
\end{ex}

\begin{ex}\label{ex.ic-3}
Suppose there is a probability distribution $\mu $ on $\mathbb{R}$ with
bounded support (i.e., $M<\infty $ such that $\mu ([-M,M]^{c})=0$) 
and that
$\{h_{i}^{n}\}$ are iid $\mu $ for $n\in \mathbb{N}$ and $i\in \lbrack n^2]$.
Define $F^{n}(x)=h_{i}^{n}$ for $x\in [i/n^2,(i+1)/n^2)$, and identify $F^{n}$ with
its periodic extension to $\mathbb{R}$. Let $\rho \geq 0$ be a smooth convolution kernel
with compact support and define%
\begin{equation*}
G^{n}(x)=\int_{\mathbb{R}}\rho (x-y)F^{n}(y)dy.
\end{equation*}%
Then Assumption \ref{assum:LDIC} holds.
Indeed, in this case one can show that $\{F^{n}\}$ satisfies the LDP on $B$
with the weak topology on $B$
and with the rate function%
\[
J(\ell)=\int_{[0,1]}L(\ell(x))dx,
\]
where $L(b)=\sup_{a\in\mathbb{R}}[ab-\log\int e^{a}\mu(da)]$. 
\footnote{Several proofs are available, including one that extends Sanov's theorem.
However, the most direct argument is to note that Mogulskii's theorem
asserts that if $Y^n(x)=\int_0^x F^n(t)dt$, then $\{Y^n\}$ satisfies an LDP in $C([0,1])$
with the rate function $I(\phi)$ equal to $\int_0^1 L(\dot{\phi}(t))dt$ if $\phi$
is absolutely continuous with $\phi(0)=0$ and $\infty$ otherwise.
Using $F^n(x)=\dot{Y}^n(x)$ a.s. in $x$ (w.p.1), we can find the result stated for the 
sequence $\{F^n\}$ using integration by parts.}
If
$\int_{\mathbb{R}}hf_{n}dx\rightarrow\int_{\mathbb{R}}hfdx$ for all $h\in \B$
then, in particular, $\int_{\mathbb{R}}\rho(x-y)f_{n}(y)dy\rightarrow
\int_{\mathbb{R}}\rho(x-y)f(y)dy$ for each $x\in\mathbb{R}$. If in addition
$\left\vert f_{n}(y)\right\vert \leq M$ for all $n\in \N$ and $y\in [0,1]$ then $\{\int_{\mathbb{R}}\rho(x-y)f_{n}(y)dy\}$ are uniformly
equicontinuous, and thus the convergence is uniform in $x$. Therefore%
\[
f\mapsto\int_{\mathbb{R}}\rho(x-y)f(y)dy
\]
is a continuous mapping from $\B\cap\{f:\left\vert f(y)\right\vert \leq M$ for
$y\in\lbrack 0,1] \}$ into itself, but with the weak topology
on the domain and
the strong topology on the range. By the contraction principle $\{G^{n}\}$
satisfies the LDP on $\B$ with rate function
\[
K(h)=\inf\left\{  J(\ell):h(x)=\int_{\mathbb{R}}\rho(x-y)\ell(y)dy\right\}  .
\]
\end{ex}

The scaling used in Example \ref{ex.ic-3} is needed so that the large deviation scaling sequence of the initial conditions matches that of the random graph.
If another scaling is used that produces a different large deviation scaling sequence,
e.g. $n^{2\alpha}$,
then when $\alpha > 1$ the rate function for the initial conditions does not appear in the rate function for $\{u^n\}$,
and from the perspective of large deviations the initial conditions are deterministic.
If however $\alpha<1$ then the scaling sequence for $\{u^n\}$ is necessarily $n^{2\alpha}$,
and the LDP for $\{u^n\}$ will not reflect the randomness of $\{X^n_{ij}\}$.
% then the slower one would dominate, in that the large deviation properties of the corresponding dynamical model would not reflect any randomness coming from the noises with the faster scaling sequence.

\section{The space of graphons}
\label{sec.graphon}
\setcounter{equation}{0}

The key ingredient in the dynamical network models
formulated in the previous section is a sequence of random adjacency
matrices $\{X^n_{ij}\}$. The corresponding kernels $H^n$ and their
(averaged) limits $W$ are called graphons in the language of the graph
theory \cite{LovGraphLim12}. Before we can formulate the LDPs for dynamical
models, we first need to understand large deviations for random graphons $%
\{H^n\}$. To this end, in this section, we review certain facts about
graphons.

Recall the collection of random variables $\{X_{ij}^{n},\; i,j\in [n]\}$
(cf. \eqref{Pedge}, \eqref{Wnij}). Given such random variables, we define $%
H^{n}:[0,1]^{2}\rightarrow\lbrack0,1]$ by 
\begin{equation}\label{eqn:defofF}
H^n=\sum_{i,j=1}^n X^n_{ij}\1_{Q^n_{ij}}.
\end{equation}
We view $\{H^{n}\}$ as taking values in $\iS$, the space of measurable
functions from $[0,1]^{2}$ to $[0,1]$. $\iS$ is equipped with the
$\infty\rightarrow 1$ distance 
\begin{equation}  \label{cut-norm}
d_{{\infty\rightarrow 1}}(f,g)=\sup_{-1\le a,b\le 1}\left\vert
\int_{[0,1]^{2}}a(t)b(s)[f(t,s)-g(t,s)]dtds\right\vert ,
\end{equation}
where $a,b:[0,1]\rightarrow\lbrack-1,1\rbrack$ are measurable functions. The 
$\infty\rightarrow 1$ distance is derived from the $L^\infty\to L^1$
operator norm
\begin{equation}  \label{operator-norm}
\|W\|_{\infty\rightarrow 1}= \sup_{-1\le a,b\le 1} \int_{[0,1]^{2}} a(x)b(y)
W(x,y) dxdy,
\end{equation}
which in turn is equivalent to the cut norm 
\begin{equation}  \label{equivalent}
\|W\|_{\square} := \sup_{S,T}\left\vert \int_{S\times T} W(x,y)
dxdy\right\vert = \sup_{0 \le a,b \le 1} \left\vert\int_{[0,1]^{2}} a(x)b(y)
W(x,y) dxdy \right\vert,
\end{equation}
where the first supremum is taken over all measurable subsets of $[0,1]$. In
particular, we have (cf. \cite[Lemma~8.11]{LovGraphLim12}) 
\begin{equation}  \label{cut-to-infty}
\|W\|_{\square}\le \|W\|_{\infty\rightarrow 1} \le 4 \|W\|_\square.
\end{equation}

$\iS^n\subset \iS$ stands for the set of piecewise constant functions with
respect to the partition $\{Q^n_{ij}\}$. Specifically, $H^n\in \iS^n$ is
constant on each $Q^n_{ij}$. The $\infty\rightarrow 1$ distance on $\iS^n$ is
equivalent to 
\begin{equation}  \label{discrete-cut-norm}
d_{\infty\rightarrow 1}^{n}(f,g)=\sup_{ a^{n},\, b^{n}}\frac{1}{n^{2}}%
\sum_{i,j=1}^{n} a_{i}^{n}b_{j}^{n}[f(i/n,j/n)-g(i/n,j/n)],
\end{equation}
where 
$a^n=(a^n_1, a_2^n,\dots, a^n_n)$, $b^n=(b^n_1, b_2^n,\dots, b^n_n)$, and
each $a^{n}_i,b^{n}_i \in [-1,1], \;i\in [n]$.

Let $\mathcal{P}$ be the set of all measure preserving bijections of $[0,1]$.
For every $\sigma \in \mathcal{P}$ and $f\in \iS$ define 
\begin{equation}
f_{\sigma }(t,s)=f(\sigma (t),\sigma (s)).  \label{permute}
\end{equation}%
This defines an equivalence relation on $\iS$. Two elements $f$ and $g$ of $\iS$
are equivalent, $f\sim g,$ if $g=f_{\sigma }$ for some $\sigma \in \mathcal{P%
}$. By identifying all elements in the same equivalence class, we obtain the
quotient space $\hat{\iS}=\iS/^{\sim }$. The distance on $\hat{\iS}$ is defined as
follows: 
\begin{equation*}
\delta _{{\infty \rightarrow 1}}(f,g)=\inf_{\sigma }d_{{\infty \rightarrow 1}%
}(f_{\sigma },g)=\inf_{\sigma }d_{{\infty \rightarrow 1}}(f,g_{\sigma })
\end{equation*}%
By the Weak Regularity Lemma, $(\hat{\iS},\delta _{{\infty \rightarrow 1}})$
is a compact metric space \cite{LovSze07}. \ 

\section{The LDPs}
\label{sec.ldp}
\setcounter{equation}{0}

% Given $W \in S$ let $\{{H}^{n}\}$ and $\{\hat{H}^{n}\}$ be defined as in the last section.
For $\hat V\in \hat{\iS}$ let  
\begin{equation}  \label{rate}
I(\hat V)=\inf_{V\in\hat{V}}\Upsilon(V,W),
\end{equation}
where $\Upsilon$ is defined by 
\begin{equation}\label{Ups}
  \begin{split}
  \Upsilon(V,W) &=\int_{[0,1]^2}  R \left( \{ V(y), 1-V(y) \} \Vert \{ W(y), 1-W(y) \}  \right) dy\\
&=\int_{[0,1]^2} \left\{  V(y) \log\left( {V(y)\over W(y)}\right) +
  \left( 1-V(y) \right) \log\left( {1-V(y)\over 1-W(y)} \right) \right\} dy,
\end{split}
\end{equation}
and $R(\theta\Vert\mu)$ is the relative entropy of probability measures $\theta$ and $\mu$, i.e., 
\begin{equation*}
R\left( \theta\left\Vert \mu\right.\right)=\int \left(\log {\frac{d\theta}{%
d\mu }}\right)d\theta
\end{equation*}
if $\theta \ll \mu$ and $R\left( \theta\left\Vert \mu\right.\right)=\infty$ otherwise.

% \begin{equation}  \label{Ups}
% \Upsilon(M,H):=\int_{[0,1]^{2}}\left( \log\left( \frac{M({y})}{H({y})}%
% \right) M({y})+\log\left( \frac{1-M({y})}{1-H({y})}\right) \left[ 1-M({y})%
% \right] \right) dy.
% \end{equation}

\begin{thm}
  \label{thm.Wrandom}
Let $\{H^n\}$  be defined by \eqref{eqn:defofF}. Then
  $\{\hat{H}^{n}\}_{n\in\mathbb{N}}$ 
satisfies the LDP
with scaling sequence $n^2$ and rate function \eqref{rate}:  $I$ has compact level sets on $\hat \iS$,
\[
\liminf_{n\rightarrow \infty}\frac{1}{n^2} \log P\{\hat{H}^{n}\in O\} \geq - \inf_{\hat{V} \in O}I(\hat{V})
\]
for open $O \subset \hat \iS$,
and 
\[
\limsup_{n\rightarrow \infty}\frac{1}{n^2} \log P\{\hat{H}^{n}\in F\} \leq - \inf_{\hat{V} \in F}I(\hat{V})
\]
for closed $F \subset \hat \iS$.
\end{thm}

We now turn to the dynamical model \eqref{KM}, \eqref{KM-ic}.
Below, it will be convenient to rewrite \eqref{KM}, \eqref{KM-ic} as
\begin{eqnarray}
\partial _{t}u^{n}(t,x) &=&f\left( u^{n}(t,x),t\right) +\int H%
^{n}(x,y)D\left( u^{n}(t,x),u^{n}(t,y)\right) dy,  \label{KMn} \\
u^{n}(0,x) &=&g^{n}(x),  \label{KMn-ic}
\end{eqnarray}%
where as before
\begin{equation}
H^{n}(x,y):=\sum_{i,j=1}^{n}X_{ij}^{n}{\mathbf{1}}_{Q_{ij}^{n}}(x,y),\quad
u^{n}(t,x):=\sum_{i=1}^{n}u_{i}^{n}(t){\mathbf{1}}_{Q_{i}^{n}}(x),\quad %
\mbox{and}\quad g^{n}(x):=\sum_{i=1}^{n}g_{i}^{n}{\mathbf{1}}_{Q_{i}^{n}}(x).
\label{step-functions}
\end{equation}%

Recall that $\B$ and $\iS$ stand for the space of initial conditions and 
the space of graphons respectively. We use the $L^{2}$%
-distance on $\B$ and the $\infty \rightarrow 1$ distance on $\iS$. Let $%
\mathcal{X}:=\iS\times \B$ endowed with the product topology. On $\B,$ $\iS,$ and $\mathcal{X}$
we define the equivalence relations: 
\begin{equation*}
  \begin{split}
    g\sim g^{\prime }&\quad \mbox{if}\quad g^{\prime }=g_{\sigma },\\
    W\sim W^{\prime }&\quad \mbox{if}\quad W^{\prime }=W_{\sigma },
    \end{split}
\end{equation*}%
and 
\begin{equation*}
(W,g)\sim (W^{\prime },g^{\prime })\quad \mbox{if}\quad W^{\prime
}=W_{\sigma }\;\&\;g^{\prime }=g_{\sigma }\quad \mbox{for some}\quad \sigma
\in \mathcal{P}.
\end{equation*}%
Define the quotient spaces $\hat{\B}=\B/^{\sim }$ and $\mathcal{\hat{X}}:=%
\mathcal{X}/^{\sim }$. The distance on $\mathcal{\hat{X}}$ is given by 
\begin{equation}
d_{\mathcal{\hat{X}}}\left( \widehat{(U,g)},\widehat{(V,h)}\right)
=\inf_{\sigma}\left\{ \Vert U_{\sigma}-V\Vert _{\infty \rightarrow 1}+\Vert g_{\sigma}-h\Vert _{L^{2}([0,1])}\right\} ,  \label{d-Xhat}
\end{equation}%
where $(U,g)\in \widehat{(U,g)}$ and $(V,h)\in \widehat{(V,h)}$ are
arbitrary representatives.

Likewise, let $\mathcal{Y}:=C([0,T],B)$ and $\mathcal{\hat{Y}}:=C([0,T],\hat{%
B})$. $\mathcal{\hat{Y}}$ is a quotient space under the following relation: 
\begin{equation*}
\mathcal{Y}\ni u\sim u^{\prime }\quad \mbox{if}\quad u^{\prime
}(t,x)=u(t,\sigma (x)),\;(t,x)\in \lbrack 0,T]\times \lbrack 0,1]
\end{equation*}%
for some $\sigma \in \mathcal{P}$. The distance on $\mathcal{\hat{Y}}$ is
given by 
\begin{equation}
d_{\mathcal{\hat{Y}}}\left( \hat{u},\hat{v}\right) =\inf_{\sigma}\Vert u_{\sigma}-v \Vert _{C(0,T;L^{2}([0,1]))},
\label{d-Yhat}
\end{equation}%
where $u\in \hat{u}$ and $v\in \hat{v}$ are arbitrary representatives.

Given $(W,g)\in \mathcal{X}$ let $u\in\mathcal{Y}$ stand for the
corresponding solution of the IVP \eqref{cKM}, \eqref{cKM-ic}. By uniqueness
of solution of the IVP \eqref{cKM}, \eqref{cKM-ic} 
\begin{equation*}
F:~\mathcal{X}\ni(W,g)\mapsto u\in \mathcal{Y}
\end{equation*}
is well--defined. Furthermore, it maps all members of a given equivalence
class of $\mathcal{X}$ to the same equivalence class of $\mathcal{Y}$: 
\begin{equation*}
F(W_\sigma, g_\sigma)=u_\sigma\quad \forall \sigma\in\mathcal{P}.
\end{equation*}
Thus, $F$ may be viewed as a map between $\mathcal{\hat X}$ and $\mathcal{%
\hat Y}$.

\begin{lem}
\label{lem.c-dep} $F:\mathcal{\hat X}\rightarrow \mathcal{\hat Y}$ is a
continuous mapping.
\end{lem}

Lemma~\ref{lem.c-dep} will be proved in Section~\ref{sec.contract}. With
Theorem~\ref{thm.Wrandom} and Lemma~\ref{lem.c-dep} in place, we use the
Contraction Principle to derive the LDP for solutions of the discrete model %
\eqref{KMn}, \eqref{KMn-ic}. In addition, Lemma~\ref{lem.c-dep} justifies
\eqref{cKM}, \eqref{cKM-ic} as a continuum limit for discrete models
\eqref{KM}, \eqref{KM-ic} on any convergent sequence of dense graphs.

We remind the reader that initial conditions are assumed to be independent of the random graph.
\begin{thm}
\label{thm.model-I} For $W\in \iS$ let $\{(H^{n},g^{n})\}$ be
a sequence of random graphons and random initial data (cf.~%
\eqref{step-functions}),
and let Assumption \ref{assum:LDIC} hold. Denote by $\{u^{n}\}$ the corresponding solutions of %
\eqref{KMn}, \eqref{KMn-ic}. Then $\{\hat{u}^{n}\}$ satisfies an LDP on $%
\mathcal{\hat{X}}$ with scaling sequence $n^2$ and the rate function
\begin{equation*}
J(\hat{u})=\inf \{I(\hat{W})+K(\hat{g}):\quad \widehat{(W,g)}=F^{-1}(\hat{u%
})\}.
\end{equation*}
\end{thm}

%     Next suppose that in addition to random graphon in \eqref{KMn}, the initial data
%      $(g_{i}^{n},\;i\in \lbrack n])$ are IID RVs with distribution $\gamma$
%    on some bounded subset of $\mathbb{R}$, and independent of $\{a_{ij}^{n}\}$.
%    We define $u^{n}$ by \eqref{step-functions}, \eqref{KMn} and \eqref{KMn-ic}.
%    We note that the rate function $I$ for $\{\tilde{W}^{n}\}$ in the space of
%    graphons will be identified later on. With regard to the initial condition $%
%   \{u^{n}(0,\cdot )\}$, we have that a large deviation theorem holds here as
%    well in the space $L^{2}([0,1])$. Let 
%   \begin{equation*}
%    C(b)=\inf \left\{ R\left( \alpha \left\Vert \gamma \right. \right) :\int
%    y\alpha (dy)=b\right\} .
%   \end{equation*}%
%    Given $\ell \in L^{2}([0,1])$, the rate function $J:L^{2}([0,1])\rightarrow
%   \lbrack 0,\infty ]$ is defined by 
%   \begin{equation*}
%    J(\ell )=\int_{[0,1]}C(\ell (x))dx.
%   \end{equation*}
%   
%    With this notation in place, we can define the rate function $K$ (should be
%    precise where $u$ takes values) for $\{u^{n}\}$. We have 
%   \begin{equation*}
%    K(u)=\inf \left\{ I(w)+J(\ell ):\eqref{cKM}\text{ and }\eqref{cKM-ic}\text{
%    hold with }W=w\text{ and }g=\ell \right\} .
%   \end{equation*}

\section{The proof of Theorem~\protect\ref{thm.Wrandom}}
\label{sec.proof}
\setcounter{equation}{0}

\subsection{The weak convergence approach}\lbl{sec.represent}
The proof Theorem~\ref{thm.Wrandom} is based on the weak convergence method of \cite{BudDup19}.
We use a ``test function'' characterization of large deviations (see \cite[Theorem 1.8]{BudDup19}). 
The proof that $I$ has compact level sets appears in the appendix.
To complete the proof of the LDP for $\{\hat{H}^{n}\}$, it is sufficient to show that for each bounded and continuous (with respect to $\delta_{{\infty\rightarrow 1}}$)
$G:\hat{\iS}\rightarrow\mathbb{R}$ ,
\[
-\frac{1}{n^{2}}\log Ee^{-n^{2}G(\hat{H}^{n})}\rightarrow\inf_{\hat{V}\in
\hat{\iS}}[I(\hat{V})+G(\hat{V})] \quad\mbox{as}\; n\to\infty .
\]

At the heart of the weak convergence approach lies the following
representation for the Laplace integrals: 
\begin{equation}  \label{represent}
-\frac{1}{n^{2}}\log Ee^{-n^{2}G(\hat{F}^{n})}= \inf E\left[ \frac{1}{n^{2}}%
R\left( \theta^{n}\left\Vert \mu^{n}\right. \right) +G(\hat{\bar{F}}^{n})%
\right] ,
\end{equation}
where $\mu^{n}$ is the product measure corresponding to $\{X_{ij}^{n}\}$ on $%
\{0,1\}^{n^{2}}$
% and the relative entropy 
% \begin{equation*}
% R\left( \theta\left\Vert \mu\right.\right)=\int \left(\log {\frac{d\theta}{%
% d\mu }}\right)d\mu.
% \end{equation*}
and the infimum in \eqref{represent} is
taken over all probability measures $\theta^{n}$ on $\{0,1\}^{n^{2}}$ 
(cf.~\cite[Proposition~2.2]{BudDup19}). 
Here $\hat{\bar{F}}^{n}$ is analogous to $\hat{F}^{n}$,
in that
\begin{equation*}
\bar{F}^{n}(y)=\bar{X}_{ij}^{n}~\ \text{for }y\in Q^n_{ij},
\end{equation*}
where $\{\bar{X}_{ij}^{n}\}$ has joint distribution $\theta^{n}$, and $\hat{%
\bar{F}}^{n}$ is the corresponding equivalence class.

By \eqref{represent}, the proof of Theorem~\ref{thm.Wrandom} is reduced to showing the convergence of variational problems: for each bounded and continuous $G$  
\begin{equation}  \label{eqn:conv}
\inf_{\theta^n} E\left[ \frac{1}{n^{2}}R\left( \theta^{n}\left\Vert
\mu^{n}\right. \right) +G(\hat{\bar{F}}^{n})\right] \rightarrow\inf_{\hat{V}%
\in\hat{\iS}}[I(\hat{V})+G(\hat{V})] \quad\mbox{as}\;n\rightarrow\infty.
\end{equation}

\subsection{A Law of Large Numbers type result}

\label{sec.LLN}

Let $a_{k},k=1,\ldots,n^{2}$ be some enumeration of the points in $%
\{1,\ldots,n\}^{2}$, and let $k(i,j)$ be defined by $a_{k(i,j)}=(i,j)$. Let $%
\bar{\theta}_{k}^{n}$ be (a version of) the conditional distribution on
variable $\bar{X}_{a_{k}}^{n}$, given $\bar{X}_{a_{s}}^{n},s<k$. Thus for $%
m=0,1$,%
\begin{equation*}
\bar{\theta}_{k}^{n}(\{m\})(\omega)=P\left\{ \bar{X}_{a_{k}}^{n}=m\left\vert 
\bar{X}_{a_{1}}^{n},\ldots,\bar{X}_{a_{k-1}}^{n}\right. \right\}(\omega) .
\end{equation*}
We can decompose $\theta^{n}$ and $\mu^{n}$ into products of these  conditional distributions, and then by the chain rule (see for example \cite[Proposition 3.1]{BudDup19}), 
\begin{equation}
\label{eqn:RER1}
E\left[ \frac{1}{n^{2}}R\left( \theta^{n}\left\Vert \mu^{n}\right. \right)
+G(\hat{\bar{H}}^{n})\right] =E\left[ \frac{1}{n^{2}}\sum_{k=1}^{n^{2}}R%
\left( \bar{\theta}_{k}^{n}\left\Vert \mu_{k}^{n}\right. \right) +G(\hat{%
\bar{H}}^{n})\right] ,
\end{equation}
where $\mu_{k}^{n}(A)=P(X_{a_{k}}^{n}\in A)$. Note that $\bar{\theta}%
_{k}^{n} $ is random and measurable with respect to $\mathcal{F}_{k-1}^{n}$,
where $\mathcal{F}_{k}^{n}=\sigma(\bar{X}_{a_{s}}^{n},s\leq k)$,
while $\mu^n_k$ is deterministic.
For an analogous calculation but with more details see \cite[Section 3.1]{BudDup19}.

We would like to relate the weak limits of $\{\hat{\bar{H}}^{n}\}$ to a
function that measures the \textquotedblleft new\textquotedblright\ link
probabilities under $\theta^{n}$, as well as the cost to produce these new
probabilities. The original probabilities are $\mu_{k}^{n}(\{1\})$, and the
new ones are $\bar{\theta}_{k}^{n}(\{1\})$. Let 
\begin{equation*}
\bar{M}^{n}({y})=\bar{\theta}_{k}^{n}(\{1\})~\ \text{if } {y}\in Q^n_{ij}.
\end{equation*}
Note that $\{\bar{M}^{n}\}$ are random variables with values in $\iS$, and and
that since $\hat{\iS}$ is compact $\{\hat{\bar{M}}^{n}\}$ and $\{\hat{\bar{H}}%
^{n}\}$ are automatically tight.

Letting 
\begin{equation*}
W^{n}({x,y})=\mu_{k}^{n}(\{1\})~\ \text{if } {(x,y)}\in Q^n_{ij},
\end{equation*}
we can write 
\begin{equation}
\label{eqn:RER2}
\frac{1}{n^{2}}\sum_{k=1}^{n^{2}}R\left( \bar{\theta}_{k}^{n}\left\Vert
\mu_{k}^{n}\right. \right) =\Upsilon(\bar{M}^{n},W^{n}),
\end{equation}
where $\Upsilon(\cdot,\cdot)$ is defined in \eqref{Ups}. Note that while $%
W^{n}$ is deterministic, $\bar{M}^{n}$ need not be. We will also want to
note that trivially
\begin{equation*}
\Upsilon(H,W)\geq\inf_{\sigma}\Upsilon(H_{\sigma},W)
\end{equation*}
for any $H,W \in \iS$.

We next state a LLN type result for the sequence of ``controlled'' random graphs $\{\bar{H}^{n}\}$.

\begin{lem}
\label{lem:conv}For any $\delta>0$ 
\begin{equation*}
P\left( d_{\infty\rightarrow 1}^{n}(\bar{H}^{n},\bar{M}^{n})\geq\delta%
\right) \rightarrow0,
\end{equation*}
and therefore for any $\delta>0$%
\begin{equation*}
P\left( \delta_{{\infty\rightarrow 1}}(\hat{\bar{H}}^{n},\hat{\bar{M}}%
^{n})\geq \delta\right) \rightarrow0.
\end{equation*}
\end{lem}

To prove Lemma \ref{lem:conv}, we use a new version of the Bernstein bound that allows
dependence between the random variables.

\begin{lem}
Let $\{Z_{i},i=1,\ldots,N\}$ be random variables with the following
properties.
\begin{enumerate}
\item $\left\vert Z_{i}\right\vert \leq c<\infty$ a.s.,

\item There is a filtration $\{\mathcal{F}_{i}\}$ such that each $Z_{j}$ for$%
\ 1\leq j<i$ is $\mathcal{F}_{i}$-measurable.
Let $m_{i}=E[Z_{i}|\mathcal{F}_{i}]$. Then for $\delta>0$%
\begin{equation*}
P\left( \frac{1}{N}\sum_{i=1}^{N}(Z_{i}-m_{i})\geq\delta\right) \leq
e^{-Nh(\delta/c)},
\end{equation*}
where $h(u)=(1+u)\log(1+u)-u>0$ for $u>0$.
\end{enumerate}
\end{lem}

\begin{proof}
Since $\left\vert Z_{i}\right\vert \leq c$, the conditional distribution of
$\left\vert Z_{i}-m_{i}\right\vert $ given $\mathcal{F}_{i}$ is also bounded
uniformly by $c$. By straightforward calculations using Taylor's theorem,%
\[
E\left[  e^{\alpha(Z_{i}-m_{i})}|\mathcal{F}_{i}\right]  \leq e^{e^{\alpha
c}-1-\alpha c}\text{ a.s.}%
\]
(the same calculation is used in the proof of the Bernstein bound). For any
$\alpha>0$%
\begin{align*}
P\left(  \frac{1}{N}\sum_{i=1}^{N}(Z_{i}-m_{i})\geq\delta\right)   &
=P\left(  e^{\alpha\sum_{i=1}^{N}(Z_{i}-m_{i})}\geq e^{N\alpha\delta}\right)
\\
&  \leq e^{-N\alpha\delta}Ee^{\alpha\sum_{i=1}^{N}(Z_{i}-m_{i})}.
\end{align*}
We then bound $Ee^{\alpha\sum_{i=1}^{N}(Z_{i}-m_{i})}$ by recurring backwards
from $i=N$:%
\begin{align*}
Ee^{\alpha\sum_{i=1}^{N}(Z_{i}-m_{i})} &  =E\left[  E\left[  \left.
e^{\alpha\sum_{i=1}^{N}(Z_{i}-m_{i})}\right\vert \mathcal{F}_{N}\right]
\right]  \\
&  =E\left[  E\left[  \left.  e^{\alpha(Z_{N}-m_{N})}\right\vert
\mathcal{F}_{N}\right]  e^{\alpha\sum_{i=1}^{N-1}(Z_{i}-m_{i})}\right]  \\
&  \leq Ee^{\alpha\sum_{i=1}^{N-1}(Z_{i}-m_{i})}e^{e^{\alpha c}-1-\alpha c}\\
&  \leq e^{N(e^{\alpha c}-1-\alpha c)}.
\end{align*}
Thus%
\[
P\left(  \frac{1}{N}\sum_{i=1}^{N}(Z_{i}-m_{i})\geq\delta\right)  \leq
e^{-N\alpha\delta}e^{N(e^{\alpha c}-1-\alpha c)}.
\]
Now optimize on $\alpha>0$. Calculus gives
$\delta-ce^{\alpha c}+c=0$,
so
\[
e^{\alpha c}=\frac{c+\delta}{c}\text{ or }\alpha=\frac{1}{c}\log\left(
\frac{c+\delta}{c}\right)  >0.
\]
This choice gives the value%
\begin{align*}
\delta\alpha-e^{\alpha c}+1+\alpha c &  =\frac{\delta}{c}\log\left(
\frac{c+\delta}{c}\right)  -\left(  \frac{c+\delta}{c}\right)  +1+\log\left(
\frac{c+\delta}{c}\right)  \\
&  =-\frac{\delta}{c}+\left(  1+\frac{\delta}{c}\right)  \log\left(
1+\frac{\delta}{c}\right)  \\
&  =h\left(  \frac{\delta}{c}\right)  .
\end{align*}
\end{proof}

\begin{proof}
[Proof of Lemma \ref{lem:conv}] We apply the previous lemma with $Z_{i}$
replaced by $a_{i}^{n}b_{j}^{n}\bar{X}_{ij}^{n}$, $m_{i}$ replaced by
$a_{i}^{n}b_{j}^{n}\bar{\theta}_{k(i,j)}^{n}(\{1\})$, $N$ replaced by $n^{2}$, and $c=1$ to get
\[
P\left(  \frac{1}{n^{2}}\sum_{i,j}^{n}a_{i}^{n}b_{j}^{n}[\bar{X}_{ij}^{n}%
-\bar{\theta}_{k(i,j)}^{n}(\{1\})]\geq\delta\right)  \leq e^{-n^{2}h(\delta)}.
\]
Since $h(\delta)>0$ for $\delta>0$, we can proceed exactly as in a
LLN argument for uncontrolled random graphs that appears in \cite[Lemma~4.1]{GueVer2016}.
Using \eqref{discrete-cut-norm} and that there are $2^n$ choices for $a^n$ and $b^n$, the union bound gives 
\[
P\left( d_{\infty\rightarrow 1}^{n}(\bar{H}^{n},\bar{M}^{n})\geq\delta%
\right)\leq 2^{n+1}e^{-n^2h(\delta)} = 2e^{n\log 2}e^{-n^2 h(\delta)} \rightarrow 0.
\]
\end{proof}

\subsection{Completion of the proof of Theorem \ref{thm.Wrandom}}

By the discussion in \S \ref{sec.represent}, it remains to show 
\begin{equation}
\lim_{n\to\infty}\inf_{\theta^{n}}E\left[ \frac{1}{n^{2}}R\left(
\theta^{n}\left\Vert \mu ^{n}\right. \right) +G(\hat{\bar{H}}^{n})\right]
=\inf_{\hat {V}\in\hat{\iS}}[I(\hat{V})+G(\hat{V})].  \label{eqn:limit}
\end{equation}

We first establish a lower bound. Let $\{\theta^{n}\}$ be any sequence for
which $\theta^{n}$ is a probability measure on $\{0,1\}^{n^{2}}$. Construct $%
\{\bar{H}^{n}\},\{\bar{M}^{n}\},\{\hat{\bar{H}}^{n}\},\{\hat{\bar{M}}^{n}\}$
and $\{W^{n}$\} as in Sections \ref{sec.graphon} and \ref{sec.LLN}, and
note that $d_{{\infty\rightarrow 1}}(W^{n},W)\rightarrow 0$. Since
$(\hat{\iS},\delta_{{\infty\rightarrow 1}})$ is compact, $\{\hat{\bar{H}}^{n}\}$ and $%
\{\hat{\bar{M}}^{n}\}$ are automatically tight. Consider any subsequence
along which $\{\hat{\bar{H}}^{n}\}$ and $\{\hat{\bar{M}}^{n}\}$ converge in
distribution, and label the limits $\hat{\bar{H}}$ and $\hat{\bar{M}}$. By
Lemma \ref{lem:conv}, $\hat{\bar{H}}=\hat{\bar{M}}$. We use Fatou's lemma,
the equations (\ref{eqn:RER1}) and (\ref{eqn:RER2}), 
and the lower semicontinuity of relative entropy
(see the proof of the lower semicontinuity of $I$ in the appendix)
along this subsequence to obtain%
\begin{align*}
& \liminf_{n\rightarrow\infty} E\left[ \frac{1}{n^{2}}R\left( \theta^{n}\left\Vert \mu
^{n}\right. \right) +G(\hat{\bar{H}}^{n})\right] \\
& \quad =\liminf_{n\rightarrow\infty} E\left[ \Upsilon(\bar{M}^{n},W^{n})+G(\hat{\bar{H}}^{n})%
\right] \\
& \quad \geq\liminf_{n\rightarrow\infty} E\left[ \inf_{V\in\hat{\bar{M}}^{n}}\Upsilon(V,W^{n})+G(%
\hat{\bar{H}}^{n})\right] \\
& \quad \geq E\left[ \inf_{V\in\hat{\bar{M}}}\Upsilon(V,W)+G(\hat{\bar{M}})%
\right] \\
& \quad\geq\inf_{\hat{V}\in\hat{\iS}}[I(\hat{V})+G(\hat{V})].
\end{align*}
Since $\{\theta^{n}\}$ is arbitrary, an argument by contradiction then gives%
\begin{equation*}
\liminf_{n\rightarrow\infty}\inf_{\theta^{n}}E\left[ \frac{1}{n^{2}}R\left(
\theta^{n}\left\Vert \mu^{n}\right. \right) +G(\hat{\bar{H}}^{n})\right]
\geq\inf_{\hat{V}\in\hat{\iS}}[I(\hat{V})+G(\hat{V})].
\end{equation*}

Next we consider the reverse bound. Let $\delta>0$ and choose $V^{\ast}\in \iS$
such that 
\begin{equation*}
\lbrack\Upsilon(V^{\ast},W)+G(\hat{V}^{\ast})]\leq\inf_{\hat{V}\in\hat{\iS}}[I(%
\hat{V})+G(\hat{V})]+\delta.
\end{equation*}
Letting $\theta^{\ast,n}$ correspond to $V^{\ast}$ in exactly the same way
that $\mu^{n}$ corresponds $W$, we can apply Lemma \ref{lem:conv} (or the
ordinary LLN) to establish that $\delta_{{\infty\rightarrow 1}}(\hat{\bar{H}}%
^{n},\hat {V}^{\ast})\rightarrow0$ in distribution. We also have by Jensen's
inequality that%
\begin{align*}
\frac{1}{n^{2}}R\left( \theta^{\ast,n}\left\Vert \mu^{n}\right. \right) & =%
\frac{1}{n^{2}}\sum_{k=1}^{n^{2}}\int_{Q_{a_{k}}}R\left( \left\{ \bar{M}^{n}({y}%
),1-\bar{M}^{n}({y})\right\} \left\Vert \left\{ W^{n}({y}),1-W^{n}({y})\right\} \right. \right) dy \\
                       & \leq\int_{\lbrack0,1]^{2}}R\left( \left\{ V^{\ast}({y}),1-V^{\ast }({y})\right\}
            \left\Vert \left\{ W({y}),1-W({y})\right\} \right. \right) dy \\
& =\Upsilon(V^{\ast},W)
\end{align*}
(the reverse bound also holds as $n\rightarrow\infty$ by lower
semicontinuity). Since we have made a particular choice of $\theta^{n}$, it
follows from the dominated convergence theorem that 
\begin{align*}
& \limsup_{n\rightarrow\infty}\inf_{\theta^{n}}E\left[ \frac{1}{n^{2}}%
R\left( \theta^{n}\left\Vert \mu^{n}\right. \right) +G(\hat{\bar{H}}^{n})%
\right] \\
& \quad \leq\limsup_{n\rightarrow\infty}E\left[ \frac{1}{n^{2}}R\left(
\theta^{\ast,n}\left\Vert \mu^{n}\right. \right) +G(\hat{\bar{H}}^{n})\right]
\\
& \quad =[\Upsilon(V^{\ast},W)+G(V^{\ast})] \\
& \quad \leq\inf_{\hat{V}\in\hat{\iS}}[I(\hat{V})+G(\hat{V})]+\delta.
\end{align*}
Letting $\delta\rightarrow0$ establishes the upper bound, and completes the
proof.

\section{Applying the Contraction Principle}
\label{sec.contract}
\setcounter{equation}{0}

In this section, we use Theorem~\ref{thm.Wrandom} and the contraction
principle to prove the LDP for dynamical model \eqref{KM}, \eqref{KM-ic}. To
this end, we need to establish continuous dependence of the solutions of the
corresponding IVPs on a kernel $W$ with respect to the cut norm and on
initial data with respect to the topology of $L^2([0,1])$.

\subsection{Proof of Lemma~\protect\ref{lem.c-dep}}

Let $U$ and $V$ be two measurable functions on $[0,1]^{2}$ with values in $%
[0,1]$ and consider the following IVPs 
\begin{eqnarray}
\partial _{t}u(t,x) &=&f\left( u(t,x),t\right) +\int U(x,y)D\left(
u(t,x),u(t,y)\right) dy,  \label{uKM} \\
u(0,x ) &=&g(x),  \label{uKM-ic}
\end{eqnarray}%
and 
\begin{eqnarray}
\partial _{t}v(t,x) &=&f\left( v(t,x),t\right) +\int V(x,y)D\left(
v(t,x),v(t,y)\right) dy,  \label{vKM} \\
v(0,x ) &=&h(x),  \label{vKM-ic}
\end{eqnarray}%
where $g,h\in L^{\infty }([0,1])$ and $x\in [0,1]$.

\begin{lem}
\label{thm.cont-model-I} For a given $T>0$, we have 
\begin{equation}  \label{cont-model-I}
\left\| u-v\right\|_{C(0,T, L^2([0,1]))}\le C \left(\|U-V\|_{\infty\to 1} +
\|g-h\|_{L^2([0,1])}\right),
\end{equation}
where $C$ depends on $T,$ but not on $U, V$ or $g, h$\footnote{%
Here and below, $C$ stands for a generic constant.}.
\end{lem}

We will need the following finite-dimensional (Galerkin) approximation of %
\eqref{uKM}, \eqref{uKM-ic} and \eqref{vKM}, \eqref{vKM-ic}, respectively: 
\begin{eqnarray}
\partial _{t}u^{n}(t,x) &=&f\left( u^{n}(t,x),t\right) +\int
U^{n}(x,y)D\left( u^{n}(t,x),u^{n}(t,y)\right) dy,  \label{uKMn} \\
u^{n}(0,x ) &=&g^{n}(x),  \label{uKMn-ic}
\end{eqnarray}%
and 
\begin{eqnarray}
\partial _{t}v^{n}(t,x) &=&f\left( v^{n}(t,x),t\right) +\int
V^{n}(x,y)D\left( v^{n}(t,x),v^{n}(t,y)\right) dy,  \label{vKMn} \\
v^{n}(0,x ) &=&h^{n}(x),  \label{vKMn-ic}
\end{eqnarray}%
where as in \eqref{step-functions} $U^{n},V^{n}$ and $g^{n},h^{n}$ stand for
the $L^{2}$--projections of $U,V$ and $g,h$ onto finite--dimensional
subspaces ${span}\{{\mathbf{1}}_{Q_{ij}^{n}}:\ (i,j)\in \lbrack n]^{2}\}$
and ${span}\{{\mathbf{1}}_{Q_{i}^{n}}:\ i\in \lbrack n]\}$ respectively:
\begin{equation}  \label{step-UV}
\begin{split}
w^n(x)=\sum_{i=1}^n w^n_{i} {\mathbf{1}}_{Q^n_{i}}(x), &\quad
w^n_{i}=n\int_{Q^n_{i}} w(x)~dx,\quad w\in\{g, h\}, \\
W^n({x,y})=\sum_{i,j=1}^n W^n_{ij} {\mathbf{1}}_{Q^n_{ij}}({x,y}), &\quad
W^n_{ij}=n^2\int_{Q^n_{ij}} W(x,y)d{x},\quad W\in\{U, V\}.
\end{split}%
\end{equation}

For solutions of the finite-dimensional models, we will need the following
lemma.

\begin{lem}
\label{lem.cont-cut} 
\begin{equation}  \label{cont-cut}
\|u^n-v^n\|_{C(0,T; L^2([0,1]))} \le C \left(\|U^n-V^n\|_{\infty\rightarrow
1}+\|g^n-h^n\|_{L^2([0,1])}\right),
\end{equation}
where $C$ is independent of $n$.
\end{lem}

The proof of Lemma~\ref{lem.cont-cut} will be presented after the proof of
Lemma~\ref{thm.cont-model-I}.

\begin{proof}[Proof of Lemma~\ref{thm.cont-model-I}.]
\begin{enumerate}
\item First, we show that
\be\lbl{first}        \|u-u^n\|_{C([0,T], L^2([0,1]))} \le C \left(\|U-U^n\|_{L^2([0,1]^2)} +\|g-g^n\|_{L^2([0,1])}\right).
\ee
To this end, let $\xi:=u-u^n,$ subtract \eqref{uKMn} from \eqref{uKM}, multiply 
the resulting equation by $\xi$ and integrate over $[0,1]$ with respect to $x$:
\be\lbl{subtract}
\begin{split}
{1\over 2} {d\over dt} \int \xi(t,x)^2dx &= \int \left[ f(u(t,x),t)-f(u^n(t,x))\right] \xi(t,x) dx\\
&\quad +\int_{[0,1]^2} U(x,y) \left\{ D(u(t,x), u(t,y)) - D(u^n(t,x), u^n(t,y))\right\} \xi(t,x) dxdy\\
&\quad + \int_{[0,1]^2} \left( U(x,y)-U^n(x,y)\right) D(u^n(t,x), u^n(t,y)) \xi(t,x) dxdy.
\end{split}
\ee
Using \eqref{Lip-f}, \eqref{Lip-D}, \eqref{D-bound},  and $|U|\le 1,$ 
% and some elementary inequalities, 
from \eqref{subtract}, we obtain
\begin{equation*}
\begin{split}
{1\over 2} {d\over dt} \int \xi(t,x)^2dx &\le \left(L_f +2L_D\right)\int \xi(t,x)^2 dx\\
&\quad + \int_{[0,1]^2} \left| U(x,y)-U^n(x,y)\right|  |\xi(t,x)| dxdy.
\end{split}
\end{equation*}
By Young's inequality, we further have
\begin{equation}\lbl{preGron}
{d\over dt} \int \xi(t,x)^2dx \le 2\left(L_f +2L_D+1/2\right)\int \xi(t,x)^2 dx+ 
\int_{[0,1]^2} \left| U(x,y)-U^n(x,y)\right|^2 dxdy.
\end{equation}
We obtain \eqref{first} from \eqref{preGron} via Gronwall's inequality.
Similarly, we have
\be\lbl{first-a}
\|v-v^n\|_{C([0,T], L^2([0,1]))} \le C \left(\|V-V^n\|_{L^2([0,1]^2)}+\|h-h^n\|_{L^2([0,1])}\right).
\ee

\item
Using contractivity of the $L^2$--projection operator with
respect to the cut norm (cf.~\cite{LovGraphLim12}), 
$
\|U^n\|_\Box\le \|U\|_\Box,
$
$
\|V^n\|_\Box\le \|V\|_\Box,
$
and \eqref{cut-to-infty}, we have
\be\lbl{contractive}
\|U^n-V^n\|_{\infty\to 1} \le 4 \|U-V\|_{\infty\to 1}.
\ee
This and Lemma~\ref{lem.cont-cut} imply
\be\lbl{finite-d}
\|u^n-v^n\|_{C([0,T], L^2([0,1]))}\le C \left( \|U-V\|_{\infty\to 1} +\|g^n-h^n\|_{L^2([0,1])}\right).
\ee
\item
From \eqref{first}, \eqref{first-a}, and \eqref{finite-d}, by the triangle inequality, we 
have
\begin{equation*}
  \begin{split}
\|u-v\|_{C([0,T], L^2([0,1]))} & \le C \left(\|U-V\|_{\infty\to 1}+\|U-U^n\|_{L^2([0,1])}+
  \|V-V^n\|_{L^2([0,1])}\right.\\
  & \left.\quad  +\|g-g^n\|_{L^2([0,1])} +\|h-h^n\|_{L^2([0,1])}+\|g^n-h^n\|_{L^2([0,1])}\right).
\end{split}
\end{equation*}
We obtain \eqref{cont-model-I} after sending $n\to\infty$. 
\end{enumerate}
\end{proof}

It remains to prove Lemma~\ref{lem.cont-cut}. We are following the lines of
the proof of Proposition~2 in \cite{OliRei19}. 
\begin{proof}[Proof of Lemma~\ref{lem.cont-cut}.]
  \begin{enumerate}
  \item
    Using the bounds on $U,V, f, D,$ 
    $$
    |U|\le 1, \; |V|\le 1,\; |f|\le 1, \; |D|\le 1,
    $$
    and the initial data
    $$
    \max\{ \|u^n(0,\cdot)\|_{L^\infty([0,1])},
    \|v^n(0,\cdot)\|_{L^\infty([0,1])}\}\le
 \max\{\|g\|_{L^\infty([0,1])},   \|h\|_{L^\infty([0,1])}\},
    $$
    it follows from \eqref{uKMn}-\eqref{vKMn-ic} that
    \be\lbl{max}
    \max_{(t,x)\in [0,T]\times [0,1]} |w^n(t,x)|\le M, \quad w^n\in \{u^n, v^n\}
    \ee
    for some $M\in (0,\infty)$ independent of $n$.
  \item
    Since $D$ is a Lipschitz continuous bounded function and $D\in H^s_{\mathrm{loc}}(\R^2), s>1,$ there is
    a Lipschitz continuous bounded function $D_M\in H^s(\R^2), s>1,$ which coincides with $D$ on the
    ball of radius $\sqrt{2}M$ centered at the origin, $B(0, \sqrt{2}M)$. Indeed,
    as $D_M$ one can take $D_M(\mathbf{x})=\xi_M(\mathbf{x}) D(\mathbf{x})$, where $\xi_M$
    is an infinitely differentiable bump function equal to $1$ on $B(0, \sqrt{2}M)$ and equal to $0$ outside of
    $B(0, 2\sqrt{2}M)$. In view of \eqref{max}, replacing $D$ with $D_M$ is not going to affect the solutions
    of the IVPs \eqref{uKMn}, \eqref{uKMn-ic} and \eqref{vKMn}, \eqref{vKMn-ic} on $[0,T]$. Thus, without loss
    of generality for the remainder of the proof we assume that $D\in H^s(\R^2), s>1$. In this case, letting
    $\phi$ be the Fourier transform of $D$, we have $\phi\in L^1(\R^2)$ and $D$ can be written as
     \be\lbl{Fourier}
     D({u})=\int_{\R^2} e^{2\pi \iu {u}\cdot z} \phi({z})dz,
     \quad {u}=(u_1, u_2), \; {z}:=(z_1, z_2), \; {u}\cdot{z}=u_1z_1+u_2z_2.
     \ee
   \item Recall that $U^n$ and $V^n$ are step functions (cf.~\eqref{step-UV}). Likewise,
  the solutions of the finite--dimensional IVPs \eqref{uKMn}, \eqref{uKMn-ic}
  can be written as
  \be\lbl{step-uv}
w^n(t,x)=\sum_{i=1}^n w^n_{i}(t)\1_{Q^n_i}(x),\quad w\in\{u, v\}.
\ee
Denote
\be\lbl{delta}
\delta^n_i(t):=u^n_i(t)-v^n_i(t),\; i\in [n].
\ee
By subtracting \eqref{vKMn} from \eqref{uKMn}, we have
\be\lbl{KM-aKM}
\begin{split}
\delta^n_i(s) = \delta^n_i(0)+ \int_0^s &  \left\{  n^{-1} \sum_{j=1}^n 
   U^n_{ij} \left( D\left(u^n_i(\tau), u^n_j(\tau)\right)-D\left(v^n_i(\tau), v^n_j(\tau)\right)
\right)\right.\\
        &\quad +\left[ f\left( u^n_i (\tau), \tau\right)-
f\left( v^n_i(\tau), \tau \right) \right]\\
&\quad +\left. n^{-1}\sum_{j=1}^n \left( U^n_{ij}-V^n_{ij}\right) D\left( v^n_i(\tau), v^n_j(\tau)\right)
\right\}
d\tau,
\end{split}
\ee
where $U^n$ and $V^n$ were defined in \eqref{step-UV}.

By continuity, there are $0\le t_i\le T$ and $\sigma_i\in\{1,-1\}$ such that 
\be\lbl{continuity}
\sup_{s\in [0,T]} \left|\delta^n_i (s)\right|=\sigma_i\delta^n_i(t_i), \; i\in [n].
\ee
Thus,
\be\lbl{estimate-Delta}
\begin{split}
  \Delta(T)&:= \int \sup_{s\in [0,T]} \left| u^n(s,x)-v^n(s,x)\right| dx=
  n^{-1} \sum_{i=1}^n \sigma_i \delta^n_i(t_i)\\
  &=n^{-1}\sum_{i=1}^n \sigma_i \delta^n_i(0)\\
&\quad + \int_0^T n^{-2} \sum_{i,j=1}^n  \sigma_i U^n_{ij} 
\left( D\left(u^n_i(\tau), u^n_j(\tau)\right)-D\left(u^n_i(\tau), u^n_j(\tau) \right) \right)
\1_{[0,t_i]}(\tau) d\tau\\
&\quad +\int_0^T n^{-1} \sum_{i=1}^n  \sigma_i \left[ f\left( u^n_i (\tau), \tau\right)-
f\left( v^n_i(\tau), \tau\right) \right] \1_{[0,t_i]}(\tau) d\tau\\
&\quad + \int_0^T n^{-2} \sum_{i,j=1}^n  \sigma_i 
\left( U^n_{ij}-V^n_{ij} \right) D\left( v^n_i(\tau), v^n_j(\tau)\right) \1_{[0,t_i]}(\tau) d\tau\\
&=n^{-1}\sum_{i=1}^n \sigma_i \delta^n_i(0) +I_1+I_2+I_3.
\end{split}
\ee

Using Lipschitz continuity of $D$ and $f$ (cf.~\eqref{Lip-D} and \eqref{Lip-f})
and the fact that $|U^n_{ij}|\le 1,$ we have
\be\lbl{I1+I2}
I_1 + I_2\le \int_0^T \left(2L_D+L_f\right)  \Delta(\tau) d\tau.
\ee
On the other hand, using \eqref{Fourier}, 
we estimate
\be\lbl{I3}
I_3 \le n^{-2} \int_0^T \int_{\R^2}  \left|\sum_{i,j}^n \left(V^n_{ij}-U^n_{ij}\right) 
e^{2\pi\iu u^n_i(\tau)z_1} e^{2\pi\iu v^n_j(\tau) z_2} \right| \phi({z}) d{z} d\tau
\ee
Decomposing $e^{2\pi\iu u^n_i(\tau)z_1}$ and $e^{2\pi\iu v^n_j(\tau)z_2}$ into sums 
of real and imaginary parts, each not exceeding $1$ in absolute value, we have
\be\lbl{finish-I3}
I_3\le  4T \|U^n-V^n\|_{\infty\to 1} \|\phi\|_{L^1(\R^2)}.
\ee
Combining \eqref{estimate-Delta}, \eqref{I1+I2}, and \eqref{finish-I3}, and 
using Gronwall's inequality and the definition of $\delta^n_i(0)$,  we obtain 
\be\lbl{T-estimate}
\begin{split}
\int \sup_{s\in [0,T]} \left| u^n(s,x)-v^n(s,x)\right| dx
& \le  e^{(2L_D+L_f)T} \left( 4T \|\phi\|_{L^1(\R^2)} \|U^n-V^n\|_{\infty\to 1} \right.\\
  &\qquad + \left. \|g^n-h^n\|_{L^2([0,1])}\right).
\end{split}
\ee
\item Using \eqref{T-estimate} and \eqref{max}, we have
  \begin{equation*}
    \begin{split}
  \sup_{t\in [0,T]}\int \left(u^n(t,x)-v^n(t,x)\right)^2 dx &\le 2M \sup_{t\in [0,T]}
  \int \left|u^n(t,x)-v^n(t,x)\right| dx\\
  & \le  2M
  \int  \sup_{t\in [0,T]} \left|u^n(t,x)-v^n(t,x)\right| dx\\
 & \le 
 2M e^{(2L_D+L_f)T}\left(4T \|\phi\|_{L^1(\R^2)} \|U^n-V^n\|_{\infty\to 1}
   +\|g_n-h_n\|_{L^2([0,1])}\right).
\end{split}
\end{equation*}
\end{enumerate}
\end{proof}

Finally, given $\widehat{(U,g)}$ and $\widehat{(V,h)}$,
% from $\mathcal{\hat X}$,
fix two representatives $(U,g) \in \widehat{(U,g)}$ and $%
(V,h) \in \widehat{(V,h)}$. Denote the corresponding solutions of the IVP and
their equivalence classes by $u,v$ and $\hat u, \hat v$ respectively.\ Using
Lemma~\ref{thm.cont-model-I}, we have 
\begin{equation*}
\begin{split}
d_{\mathcal{\hat Y}}(\hat u, \hat v) & =\inf_{\sigma} \|
u_{\sigma}-v\|_{C(0,T; L^2([0,1]))} \\
& \le C \inf_{\sigma} \left\{
\|U_{\sigma}-V\|_{\infty\to 1} +
\|g_{\sigma}-h\|_{L^2([0,1])}\right\} \\
& \le C d_{\mathcal{\hat{X}}} \left( \widehat{(U,g)}, \widehat{(V,h)}\right).
\end{split}%
\end{equation*}
This shows the continuity of $F:\mathcal{\hat X}\to\mathcal{\hat Y}$
needed for the application of the contraction principle to the dynamical
model at hand.

\section{Generalizations}
\label{sec.generalize}
\setcounter{equation}{0}
In this section, we describe two generalizations of the analysis in
the main part of the paper. First, we extend the LDP to cover the
original model \eqref{KM} with random parameters.
Second, we discuss the case of the dynamical model on a sequence
of sparse graphs. The analysis in the previous sections suggests a natural
extension of the LDP derived for the dense networks to their sparse
counterparts. To explain this extension, we formulate the dynamical
model on a convergent sequence of sparse W-random graphs \cite{BCCZ19}. Next, we prove an
LDP for sparse W-random graphs in the space of nonnegative finite measures
with the vague topology. We conjecture that this LDP can be upgraded
to the LDP with the same rate function in the space of graphons with 
the cut norm topology, which would afford further application to the
dynamical problem. We support this conjecture by demonstrating the
key estimate needed for the proof of the lower large deviations bound
and outlining the steps needed for the proof of the upper bound.
The latter however leads to new technical difficulties, which will be
addressed elsewhere.

\subsection{Random parameters}

We now revisit \eqref{frame}, \eqref{frame-ic} to address the dependence of $f$ on random parameters.
To this end, we rewrite \eqref{frame}, \eqref{frame-ic} as follows
\begin{equation}\label{newframe}
  \begin{split}
    \dot u^n_i &= f(u^n_i,\eta^n_i,t) +\frac{1}{n} \sum_{j=1}^n X^n_{ij} D( u^n_i, u^n_j),\;\dot \eta^n_i = 0,\\
    u^n_i(0)&=g_i^n,\;
    \eta_i^n(0)=\xi^n_i,\quad i\in [n],
  \end{split}
\end{equation}
where $\xi^n_i\in\R^d$ is a random array. Thus, the random parameters can be treated in the same
way as the initial data.

We formulate the assumptions on $\{\xi^n_i\}$ in analogy to how this was done for $\{ g_i^n \}$ in
Section~\ref{sec.model}. Specifically, let $\{  J^n \}$ be a sequence of iid ${\B}^d$-valued random variables
independent from $\{X_{ij}^n\}$ and $\{G^n \}$. Then
\begin{equation*}
\xi_{i}^{n}=n\int_{Q_{i}^{n}} J^{n}(y)dy,\quad i\in [n] 
\end{equation*}
and
$$
\bar{J}^{n}(x)=\xi_{i}^{n}\;\mbox{for}\;x\in Q_{i}^{n}.
$$
In analogy to Assumption~\ref{assum:LDIC}, we impose the following.
\begin{assume}
  \label{assum:LDIC+}
  $\{\bar{J}^{n}\}$ satisfies the LDP in ${\B}^d$ with the rate function $L$ and scaling sequence $n^2$.
\end{assume}

All other assumptions on the data in \eqref{newframe} remain the same with one exception:
the Lipschitz condition \eqref{Lip-f} is replaced by the condition
\begin{equation*}
  |f(u,\xi, t)-f(u^\prime, \xi^\prime,t)|\leq L_{f}\left(|u-u^\prime|+|\xi-\xi^\prime|\right)\
  \quad u, u^\prime \in {\mathbb{R}},\;\xi, \xi^\prime\in {\mathbb{R}^d},\; t\ge 0.
\end{equation*}

The continuum limit for \eqref{newframe} is given by
\begin{equation}\label{cframe}
  \begin{split}
    \partial_t u(t,x)& =f(u,j(x),t) + \int W(x,y)D\left( u(t,x), u(t,y)\right) dy,\\
    u(0,x)&=g(x),
  \end{split}
  \end{equation}
  where $g\in\B$ and $j\in\B^d$. The initial value problem \eqref{cframe} has a unique solution
  $u\in \mathcal{Y}$. (Recall that $\mathcal{Y}$ stands for $C([0,T], L^2([0,1]))$.)
  Furthermore, we can formulate a counterpart of Lemma~\ref{lem.c-dep} for the model at hand.
  Specifically, let
  \begin{equation*}
(W,g, j)\sim (W^{\prime },g^{\prime }, j^\prime)\quad \mbox{if}\quad W^{\prime
}=W_{\sigma },\;g^{\prime }=g_{\sigma },\;\&\; j^\prime=j_\sigma\quad \mbox{for some}\quad \sigma
\in \mathcal{P}.
\end{equation*}%
and redefine $\mathcal{X}:=\iS\times \B\times \B^d$ and $\hat{\mathcal{X}}=\mathcal{X}/^{\sim }$.
Let
$$
F:\; \mathcal{X}\ni (W,g,j)\mapsto u \in \mathcal{Y}
$$ denote the
map between the data and the solution of the initial value problem \eqref{newframe}.
As before,
$$
F(W_\sigma, g_\sigma, j_\sigma)=u_\sigma\quad \forall \sigma\in \mathcal{P}.
$$
Thus, $F: \hat{\mathcal{X}}\to \hat{\mathcal{Y}}$ is well defined. The analysis of Section~\ref{sec.contract}
with straightforward modifications then implies that that $F$ is a continuous mapping. Here, the metric in
$\hat{\mathcal{Y}}$ is unchanged and
\begin{equation}
d_{\mathcal{\hat{X}}}\left( \widehat{(U,g, j)},\widehat{(U^\prime,g^\prime, j^\prime)}\right)
=\inf_{\sigma}\left\{ \Vert U_{\sigma}-U^\prime\Vert _{\infty \rightarrow 1}+
  \Vert g_{\sigma}-g^\prime\Vert _{\B}
  +\Vert j_{\sigma}-j^\prime\Vert _{\B^d}
\right\} ,  \label{d-Xhat}
\end{equation}%
where $(U,g, j)\in \widehat{(U,g,j)}$ and $(U^\prime,g^\prime, j^\prime)\in \widehat{(U^\prime,g^\prime,j^\prime)}$ are
arbitrary representatives.

This leads to the following.
\begin{thm}
\label{thm.model-II} For $W\in \iS$ let $(H^{n},g^{n}, J^{n})$ be
a sequence of random graphons and random initial data and parameters
and let Assumptions \ref{assum:LDIC}  and  \ref{assum:LDIC+} hold.
Denote by $u^{n}$ the corresponding solutions of %
\eqref{newframe}. Then $\{\hat{u}^{n}\}$ satisfies an LDP on $%
\mathcal{\hat{X}}$ with scaling sequence $n^2$ and the rate function
\begin{equation*}
\mathbf{J}(\hat{u})=\inf \{I(\hat{W})+K(\hat{g}) + L(\hat{j}):\quad \widehat{(W,g,j)}=F^{-1}(\hat{u})\}.
\end{equation*}
\end{thm}

\subsection{Sparsity}
 Let $W:[0,1]^2\to [0,1]$ as before and let
$0<\alpha_n\le 1,\; n\in N,$ be a nonincreasing sequence.
If $\alpha_n\to 0$, we in addition assume that
\begin{equation}\label{unbdd-deg}
  \alpha_n^2n\to\infty.
\end{equation}

Define $\{\Gn\}$ by
\be\lbl{newPedge}
\P(X^n_{ij}=1)=\alpha_n W^n_{ij}.
\ee
If $\alpha_n\to\alpha_\infty>0$, $\{\Gn\}$ is a sequence of dense graphs as before.
If $\alpha_n\searrow 0$ then graphs $\{\Gn\}$ are sparse.
\begin{ex}
  Sparse Erd\H{o}s--R{\' e}nyi graphs: $W\equiv 1,$ $\alpha_n\searrow 0$. Dynamical models on sparse
  Erd\H{o}s--R{\' e}nyi graphs were studied in \cite{CopDieGia20, OliRei19}.
\end{ex}

On sparse graphs $\{\Gn\}$, the dynamical model takes the form
\begin{equation}\label{newKM}
  \dot u_i^n=f(u^n_i,t)+(\alpha_n n)^{-1} \sum_{j=1}^n X^n_{ij}D(u_i^n, u^n_j).
\end{equation}
The new scaling of the interaction term, $(\alpha_nn)^{-1}$, is used to account for sparsity.

To apply the analysis of the large deviations in the main part of the paper to the
model at hand, note that
$$
(\alpha_n n)^{-1} \sum_{j=1}^n X^n_{ij}D(u_i^n, u^n_j)=n^{-1} \sum_{j=1}^n \mathbf{X}^n_{ij}D(u_i^n, u^n_j),
$$
where
$$
\mathbf{X}^n_{ij} =\alpha_n^{-1} X^n_{ij}, \; \E\mathbf{X}^n_{ij}=W^n_{ij}.
$$
This suggests that in the sparse case one needs to study rescaled random variables
$$
\mathbf{H}^n=\alpha_n^{-1} H^n.
$$ 
By considering $\mathbf{H}^n(y)dy$ we can view $\mathbf{H}^n$ as taking values is the set of non-negative finite measures on $[0,1]^2$ with the vague topology.
This topology can be metrized so that the space is a Polish space \cite[Section A.4.1]{BudDup19}.
When viewed this way, $\mathbf{H}^n$ has the same large deviation properties as $\mathcal{N}^{n}/n^2\alpha_n$, where $\mathcal{N}^n$ is a Poisson random measure with 
intensity $\alpha_n n^2 W(y)dy$.

%Another difference is due to the scaling of the relative entropy costs.
%Expanding relative entropy for $r,b>0$ gives (up to vanishing errors) gives
%\[
%R\left(  (b\alpha_{n},1-b\alpha_{n})\left\Vert
%(r\alpha_{n},1-r\alpha_{n})\right.  \right) \approx
%n^{2}\alpha_{n}r\ell\left(  \frac{b}{r}\right)
%\]
%where $\ell\left(  z\right)  =z\log z-z+1$ for $z\geq0$.
This leads to the following statement.
\begin{thm}
  \label{thm.sparse}
   Let  $\ell\left(  z\right)  =z\log z-z+1$ for $z\geq0$ and
  suppose \eqref{unbdd-deg} holds. Consider a sequence of sparse
  W-random graphs defined by \eqref{newPedge}.  Then $\{{\mathbf{H}}^{n}\}_{n\in\mathbf{N}}$ as defined above 
  satisfies the LDP with rate function 
 \begin{equation}\label{new-rate}
 \int_{[0,1]^{2}}W(y)\ell\left(  \frac{V(y)}{W(y)}\right)  dy
 \end{equation}
 and the scaling sequence $n^2\alpha_n$.
\end{thm}

However, 
we would like to strengthen this to the cut-norm topology. 
One direction is straightforward, in that  we can still use Bernstein's bound for the rescaled array, and thereby establish the large deviation lower bound in 
the stronger topology.
For example,
if we want to compare $\mathbf{H}^n$ and $W^n$ as would be needed to establish the LLN in the stronger topology,
we find

\begin{align*}
\mathbb{P}\left(  d_{\infty\rightarrow1}(\boldsymbol{H}^{n},W^{n})\geq
\delta\right)    & \leq\sup_{a^{n},b^{n}}\mathbb{P}\left(  \frac{1}{n^2}\sum_{i,j=1}%
^{n}a_{i}^{n}b_{j}^{n}\left[  \frac{a_{ij}}{\alpha_{n}}-W_{ij}^{n}\right]
\geq\delta\right)  \\
& =\sup_{a^{n},b^{n}}\mathbb{P}\left( \frac{1}{n^2} \sum_{i,j=1}^{n}a_{i}^{n}b_{j}%
^{n}\left[  a_{ij}-\alpha_{n}W_{ij}^{n}\right]  \geq\alpha_{n}\delta\right)
\\
& \leq e^{n\log2}e^{-n^{2}h(\alpha_{n}\delta)}\rightarrow0
\end{align*}
owing to our assumption \eqref{unbdd-deg} and $h(\alpha_n \delta) \approx \alpha^2_n\delta^2/2$.
The analogous estimate as needed to establish the large deviation
estimate in the cut norm also holds.
% and we conjecture that this is sufficient to prove the LDP for the dynamical model \eqref{newKM}.

For the large deviation upper bound it was essential to work with the equivalence class,
and there it was crucial that the set $\hat \iS$ was compact. 
We conjecture than an analogous compactness holds here as well. 
Specifically, let 
\[
 I(\hat{V})= \inf_{V\in \hat{V}} \int_{[0,1]^{2}}W(y)\ell\left(  \frac{V(y)}{W(y)}\right)  dy.
 \]
Then we conjecture that under reasonable conditions on $W$ the superlinear growth of $\ell$ implies level sets of  $I(\hat{V})$ are compact in 
the natural generalization of $\hat \iS$,
and that this together with the Bernstein's bound suffices to establish the upper bound.

\section{Appendix: Proof of Lower Semicontinuity of $I$}

We want to prove that
\[
\liminf_{n\rightarrow\infty}I(\hat{V}^{n})\geq I(\hat{V})
\]
when $\hat{V}^{n}\rightarrow\hat{V}$. The latter means $V^{n}\rightarrow V$ in
$d_{\infty\rightarrow1}$. Then we have to show that%
\[
\liminf_{n\rightarrow\infty}\inf_{V\in\hat{V}^{n}}\Upsilon(V,W)\geq\inf
_{V\in\hat{V}}\Upsilon(V,W),
\]
where
\[
\Upsilon(V,W)=\int_{[0,1]^{2}}G(V(\boldsymbol{x}),W(\boldsymbol{x}%
))d\boldsymbol{x},\quad G(v,w)=v\log\left(  \frac{v}{w}\right)  +(1-v)\log
\left(  \frac{1-v}{1-w}\right)  .
\]
To simplify the proof we assume that $W$ is continuous. If $W$ is just
measurable but bounded away from $0$ and $1$ this can be justified by Lusin's Theorem.
If $W$ is not bounded away from $0$ and $1$ then we can replace $W$ by $(W\vee \delta)\wedge (1-\delta)$ for $\delta >0$ and using a similar argument to that used below send 
$\delta \rightarrow 0$.
\begin{figure}[hbt!]
 \centering
 \includegraphics[scale=.7]{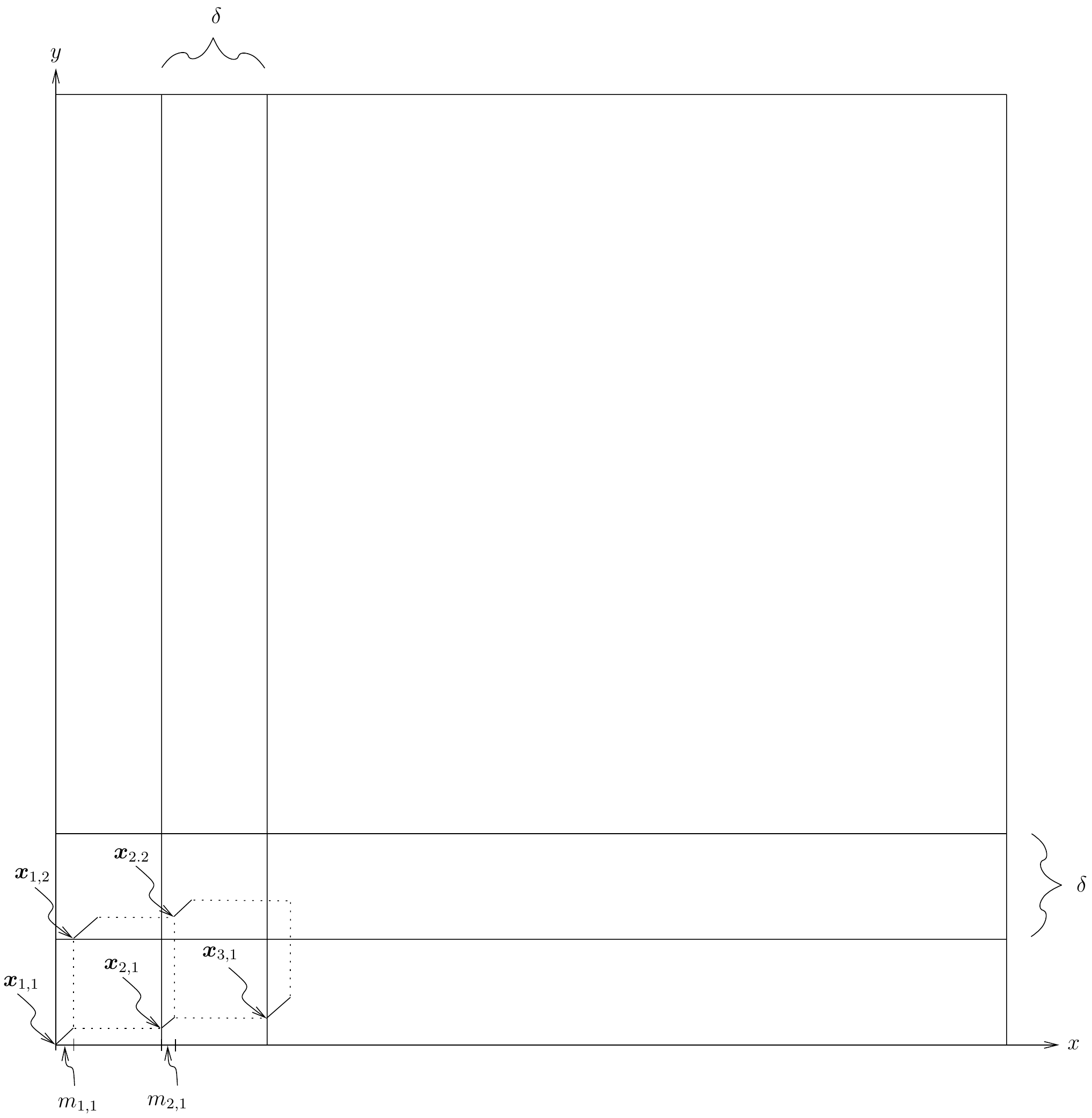}
 \caption{Construction of $\sigma$.}
\end{figure}

The proof will be based on weak convergence and a construction analogous to the ``chattering lemma'' of control theory. 
Let $\bar{V}^{n}$ come within $1/n$ of the
infimum in $\inf_{V\in\hat{V}^{n}}\Upsilon(V,W)$. Then there is $\sigma^{n}%
\in\mathcal{P}$ such that
\[
\Upsilon(\bar{V}^{n},W)=\Upsilon(V^{n},W\circ\sigma^{n}),
\]
and it is enough to show the following. Let any subsequence of $n$ be given
and consider a further subsequence (say $\bar{n}$). Then given $\varepsilon>0$
we can find $\sigma\in\mathcal{P}$ such that
\[
\liminf_{\bar{n}\rightarrow\infty}\Upsilon(V^{\bar{n}},W\circ\sigma^{\bar{n}%
})\geq\Upsilon(V,W\circ\sigma)-\varepsilon.
\]
To simplify notation we write $n$ rather than $\bar{n}$.

We define probability measures $\{\mu^{n}\}$ on $[0,1]^{5}$ by
\[
\mu^{n}(A_{1}\times A_{2}\times\cdots\times A_{5})=\int_{A_{2}\times A_{3}%
}1_{A_{1}}(V^{n}(x_{1},x_{2}))1_{A_{4}}(\sigma^{n}(x_{1}))dx_{1}1_{A_{5}%
}(\sigma^{n}(x_{2}))dx_{2}.
\]
By compactness we can assume that for the subsubsequence these converge weakly
with limit $\mu$. Also, we can write
\begin{align*}
\Upsilon(V^{n},W\circ\sigma^{n}) &  =\int_{[0,1]^{2}}G(V(x_{1},x_{2}%
),W(\sigma^{n}(x_{1}),\sigma^{n}(x_{2})))dx_{1}dx_{2}\\
&  =\int_{[0,1]^{5}}G(v,W(y_{1},y_{2}))\mu^{n}(dv\times dx_{1}\times
dx_{2}\times dy_{1}\times dy_{2}),
\end{align*}
and using the properties of $G$ (bounded and lsc) and $W$ (bounded and
continuous)%
\[
\liminf_{n\rightarrow\infty}\Upsilon(V^{n},W\circ\sigma^{n})\geq\int
_{\lbrack0,1]^{5}}G(v,W(y_{1},y_{2}))\mu(dv\times dx_{1}\times dx_{2}\times
dy_{1}\times dy_{2}).
\]
Given $\varepsilon>0$ we need to construct $\sigma\in\mathcal{P}$ such that
\[
\int_{\lbrack0,1]^{5}}G(v,W(y_{1},y_{2}))\mu(dv\times dx_{1}\times
dx_{2}\times dy_{1}\times dy_{2})\geq\Upsilon(V,W\circ\sigma)-\varepsilon.
\]

With subscripts denoting marginal distributions, it is clear that
\[
\mu_{2,4}(dx_{1}\times dy_{1})=\mu_{3,5}(dx_{2}\times dy_{2})
\]
and, since each $\sigma^{n}$ is a measure preserving bijection, that
\[
\mu_{2}(dx_{1})=dx_{1},\mu_{3}(dx_{2})=dx_{2},\mu_{4}(dy_{1})=dy_{1},\mu
_{5}(dy_{2})=dy_{2}.
\]
Thus both marginals of $\mu_{2,4}$ (and $\mu_{3,5}$) are Lebesgue measure, but
we do not know that $\mu_{2,4}$ is the measure induced by a measure preserving
bijection $\sigma$. We will approximate $\mu_{2,4}$ to construct $\sigma$, and
in doing so incur a small error in the integral which will be smaller than
$\varepsilon$.  

Let $\nu(dx\times dy)=\mu_{2,4}(dx\times dy)$. Then it suffices to find a
sequence $\theta_{k}\in\mathcal{P}$ such that if $\nu_{k}(A_{1}\times
A_{2})=\int_{A_{1}}1_{A_{2}}(\theta_{k}(x))dx$, then $\nu_{k}$ converges to
$\nu$ in the weak topology. The construction is as follows. Let $\delta=1/k$.
Then we partition $[0,1]^{2}$ according to
\[
T_{i,j}^{k}=[(i-1)\delta,i\delta)\times\lbrack(j-1)\delta,j\delta),\quad1\leq
i,j\leq k
\]
and define $m_{i,j}^{k}=\nu(T_{i,j}^{k})$. Let
\[
S(\boldsymbol{x},a)=\{\boldsymbol{x}+t(1,1):0\leq t<a\}.
\]
The graph $G\subset\lbrack0,1]^{2}$ of $\theta_{k}$ is constructed recursively
as follows.

Let $j=1$, and set $\boldsymbol{x}_{1,1}=(0,0)$. Then set $G_{1,1}%
=S(\boldsymbol{x}_{1,1},m_{1,1})$, and define $\boldsymbol{x}_{2,1}%
=(\delta,m_{1,1})$. We then iterate, setting
\[
G_{i+1,1}=G_{i,1}\cup S(\boldsymbol{x}_{i,1},m_{i,1})\text{ and }%
\boldsymbol{x}_{i+1,1}=\left(  i\delta,\sum_{r=1}^{i}m_{r,1}\right)
\]
until $i=k-1$. This assigns all the mass of $\nu([0,1)\times\lbrack
0,\delta))=\delta$ to nearby points consistent with a piecewise continuous
measure preserving bijection. Specifically, the projection of $G_{k,1}$ onto
the $y$-axis gives the set $[0,\delta)$.

Next consider $j=2$. To maintain that the graph generate a measure preserving
bijection, we now start with $\boldsymbol{x}_{1,2}=(m_{1,1},\delta)$. The
iteration is now
\[
G_{i+1,2}=G_{i,2}\cup S(\boldsymbol{x}_{i,2},m_{i,2})\text{ and }%
\boldsymbol{x}_{i+1,1}=\left(  i\delta+m_{i,1},\sum_{r=1}^{i}m_{r,2}\right)
.
\]
For $1<j\leq k$ the definitions are $\boldsymbol{x}_{1,j}=(\sum_{l=1}%
^{j}m_{1,l},j\delta)$ and
\[
G_{i+1,j}=G_{i,j}\cup S(\boldsymbol{x}_{i,j},m_{i,j})\text{ and }%
\boldsymbol{x}_{i+1,j}=\left(  i\delta+\sum_{l=1}^{j}m_{i,l},\sum_{r=1}%
^{i}m_{r,j}\right)  .
\]
See the figure below. Finally we set $G=G_{k,k}$. This graph defines an element
$\theta_{k}$ of $\mathcal{P}$. If $\nu_{k}(A_{1}\times A_{2})=\int_{A_{1}%
}1_{A_{2}}(\theta_{k}(x))dx$, then all mass inside $T_{i,j}^{k}$ under $\nu$
has stayed inside $T_{i,j}^{k}$. As a consequence $\nu_{k}$ converges weakly
to $\nu$, and the proof is complete.

\bibliographystyle{amsplain}
%\bibliography{cut}
\def\cprime{$'$} \def\cprime{$'$} \def\cprime{$'$}
\providecommand{\bysame}{\leavevmode\hbox to3em{\hrulefill}\thinspace}
\providecommand{\MR}{\relax\ifhmode\unskip\space\fi MR }
\providecommand{\MRhref}[2]{%
  \href{http://www.ams.org/mathscinet-getitem?mr=#1}{#2}
}
\providecommand{\href}[2]{#2}

\end{document}